\documentclass{article}

\usepackage{amsmath,amssymb,amsfonts,amsthm,extarrows}
\usepackage{mathtools}
\usepackage{enumitem}
\usepackage{stmaryrd}
\usepackage{tikz-cd} 
\usepackage{bbm}
\usepackage[english]{babel}
\usepackage[nottoc]{tocbibind}

\usepackage{blindtext}
\title{Stability of elliptic Fargues--Scholze $L$-packets}
\author{Chenji Fu}
\date{}

\usepackage{fancyhdr}

\usepackage{color}

\usepackage{nameref}

\usepackage{graphicx}
\graphicspath{ {./images/} }

\usepackage{soul}

\usepackage[colorinlistoftodos,textsize=footnotesize]{todonotes}
\usepackage[colorlinks]{hyperref}

\newtheorem{theorem}{Theorem}[subsection]
\newtheorem{remark}[theorem]{Remark}
\newtheorem{lemma}[theorem]{Lemma}

\newtheorem{proposition}[theorem]{Proposition}
\newtheorem{definition}[theorem]{Definition}
\newtheorem{corollary}[theorem]{Corollary}

\newtheorem{conjecture}[theorem]{Conjecture}
\newtheorem{assumption}[theorem]{Assumption}

\newcommand{\Hom}{\operatorname{Hom}}
\newcommand{\Rep}{\operatorname{Rep}}
\newcommand{\End}{\operatorname{End}}

\newcommand{\Irr}{\operatorname{Irr}}

\newcommand{\Perf}{\operatorname{Perf}}
\newcommand{\Spec}{\operatorname{Spec}}

\newcommand{\Gal}{\operatorname{Gal}}
\newcommand{\Ql}{{\overline{\mathbb{Q}}_{\ell}}}

\newcommand{\Cent}{\operatorname{Cent}}
\newcommand{\ZG}{{Z(\widehat{G})^{\Gamma}}}
\newcommand{\Oo}{{\mathcal{O}(S_{\varphi})_0}}
\newcommand{\elliptic}{\operatorname{ell}}
\newcommand{\Rel}{\operatorname{Rel}}
\newcommand{\character}{\operatorname{Char}}
\newcommand{\basic}{\operatorname{basic}}
\newcommand{\Bun}{\operatorname{Bun}}
\newcommand{\OS}{{\mathcal{O}(S_{\varphi})}}
\newcommand{\Mant}{\operatorname{Mant}}

\newcommand{\inv}{\operatorname{inv}}
\newcommand{\sr}{\operatorname{sr}}

\newcommand{\LPhi}{{\Lambda_{\widehat{\Phi}}}}
\newcommand{\tr}{\operatorname{tr}}

\newcommand{\OSZ}{{\mathcal{O}\left(S_{\varphi}/Z(\widehat{G})^{\Gamma}\right)}}
\newcommand{\rs}{\operatorname{rs}}
\newcommand{\pr}{\operatorname{pr}}
\newcommand{\ULA}{\operatorname{ULA}}
\newcommand{\id}{\operatorname{id}}
\newcommand{\FS}{\operatorname{FS}}

\newcommand{\ZoneG}{{Z^1(W_F, \widehat{G})/\widehat{G}}}

\newcommand{\Tuniv}{\mathbb{T}_{\operatorname{univ}}}

\begin{document}

	\maketitle

	\begin{abstract}
		Let $F$ be a non-archimedean local field. Let $\overline{F}$ be an algebraic closure of $F$. Let $G$ be a connected reductive group over $F$. Let $\varphi$ be an elliptic $L$-parameter. 
		For every irreducible representation $\pi$ of $G(F)$ with Fargues--Scholze $L$-parameter $\varphi$, 
		we prove that there exists a finite set of irreducible representations $\{\pi_i\}_{i \in I}$ containing $\pi$, such that $\pi_i$ has Fargues--Scholze $L$-parameter $\varphi$ for all $i \in I$ and a certain non-zero $\mathbb{Z}$-linear combination $\Theta_{\pi_0}$ of the Harish-Chandra characters of $\{\pi_i\}_{i \in I}$ is stable under $G(\overline{F})$ conjugation, as a function on the elliptic regular semisimple elements of $G(F)$. Moreover, if $F$ has characteristic zero, $\Theta_{\pi_0}$ is a non-zero stable distribution on $G(F)$.
	\end{abstract}

	\tableofcontents

	\section{Introduction}
	The local Langlands correspondence conjecturally partitions the irreducible representations of a $p$-adic group into the so-called $L$-packets. Such a partition is conjecturally characterized by the stability condition, which has been proven in many cases using the theory of endoscopy. The aim of this paper is to approach the stability from a new point of view using \cite{fargues2021geometrization}, which geometrizes the representations as sheaves on $\Bun_G$. Given an irreducible representation of $G(F)$, we define a sheaf  on $\Bun_G$ by averaging over the automorphisms of the corresponding $L$-parameter, which we observe is a Hecke eigensheaf.
	Combining this with a formula of \cite{hansen2022kottwitz}, we reduce the problem of stability to showing equi-distribution properties of the weight multiplicities of highest weight representations of an algebraic group. Our proof of equi-distribution properties might be of independent interest (Subsection \ref{Section_equi_dist}).
	
	To explain our results in more details, let $F$ be a non-archimedean local field and $G$ be a connected reductive $F$-group with Langlands dual group $\widehat{G}$ over $\Ql$, which comes canonically with a torus and a Borel $\widehat{T} \subseteq \widehat{B} \subseteq \widehat{G}$. Let $\Tuniv$ be the universal Cartan of $G$ (see \cite[Section VI. 11]{fargues2021geometrization}). The cocharacter lattice of $\Tuniv$ is canonically isomorphic to the character lattice of $\widehat{T}$:
	$$X_*(\Tuniv) \cong X^*(\widehat{T}).$$ To state our theorem for general $G$, we need to take care of all the extended pure inner forms 
	of $G$. For simplicity, let us make the following assumption in this introduction. 
	\begin{assumption}\label{Assumption_intro}
		In this introduction, assume that $G$ is $F$-split, semisimple, and simply connected.
	\end{assumption}
	
	Let $W_F$ be the Weil group of $F$. 
	The local Langlands conjecture predicts that there exists a surjective finite-to-one map from the set of  equivalence classes of irreducible  smooth representations of $G(F)$ to the set of \emph{$L$-parameters}, i.e., $\widehat{G}$-conjugacy classes of $\ell$-adically continuous $1$-cocycles $\varphi: W_F \to \widehat{G}(\Ql)$ (\cite[Section 1.1]{dhkm2020moduli}). The fiber over $\varphi$ is called an \emph{$L$-packet}, denoted by $\Pi_{\varphi}(G)$. This conjectural association is expected to satisfy a list of properties, one of which is the stability of $L$-packets.
	
	\begin{conjecture}[{\cite[Conjecture 2.2]{kaletha2022representations}, see \cite[Conjecture 6.8]{kalethalocal} for a more precise form}]\label{Conj_stab}
		Any discrete series $L$-packet is \textbf{stable}, i.e., there exists a linear combination of the Harish-Chandra characters of its members that is a stable distribution.
	\end{conjecture}
	
	Explicit constructions of the local Langlands correspondence are known in many special cases and Conjecture \ref{Conj_stab} is proven in some of those cases (for example,
	\cite{arthur2013endoscopic},
	 \cite{adler2009supercuspidal}, \cite{debacker2009depth},
	\cite{debacker2018stability},
	\cite{fintzen2023twisted},
	\cite{gan2011local},
	 \cite{kaletha2019regular}, 
	 \cite{kazhdan2006endoscopic},
	 \cite{spice2018explicit} and \cite{spice2021explicitii}). All these proofs rely on some deep results from the theory of endoscopy.

	Using a geometric approach, Fargues--Scholze \cite{fargues2021geometrization} attached a semisimplified $L$-parameter $\varphi_{\pi'}^{\FS}$ to any smooth irreducible representation $\pi'$ of $G(F)$, giving a general candidate for the local Langlands correspondence. For a semisimplified $L$-parameter $\varphi$, we define the \emph{Fargues--Scholze $L$-packet} of $\varphi$ to be
	\begin{equation}\label{Eq_FS_L-pacekt}
		\Pi_{\varphi}^{\operatorname{FS}}(G):=\{\pi' \in \Irr_{\Ql}G(F)\;|\; \varphi_{\pi'}^{\operatorname{FS}}=\varphi\}.
	\end{equation}

	Moreover, they constructed the so-called \emph{spectral action} (\cite[Chapter X]{fargues2021geometrization}), denoted by $*$. Let $\varphi$ be an elliptic $L$-parameter (Definition \ref{Def:elliptic}). The spectral action in particular gives an action of the derived category of perfect complexes of $S_{\varphi}$-representations on the derived category of $G(F)$-representations, where $S_{\varphi}$ is the centralizer of $\varphi$.  For every $\pi \in \Irr_{\Ql}G(F)$ such that $\varphi_{\pi}^{\FS}=\varphi$ (note that it is not known whether such $\pi$ exists or not),
	let us define 
	$$\pi_0:=\OS*\pi,$$ 
	where $\OS$ is the regular representation of $S_{\varphi}$ (see (\ref{Eq_def_pi}) for the definition of $\pi_0$ for general $G$). $\pi_{0}$ is an object in the derived category of $G(F)$-representations, which turns out to be a finite direct sum of irreducible representations with Fargues--Scholze $L$-parameter $\varphi$ up to degree shifts 
	(Proposition \ref{Proposition_elliptic}). We define its \emph{Harish-Chandra character} $\Theta_{\pi_0}$ as the alternating sum of the Harish-Chandra characters of the cohomologies. $\Theta_{\pi_0}$ is thus a $\mathbb{Z}$-linear combination of Harish-Chandra characters of members of the Fargues--Scholze $L$-packet of $\varphi$, as required in Conjecture \ref{Conj_stab}. Moreover, we show that $\Theta_{\pi}$ occurs in $\Theta_{\pi_{0}}$ with a non-zero coefficient (Lemma \ref{Lem_nonzero}), hence $\Theta_{\pi_{0}}$ is non-zero.
	
	The main result we prove in our paper is:
	\begin{theorem}\label{Thm_main}(Theorem \ref{Thm_main_general} and \ref{Thm_main_distribution})
		Let $\varphi$ be an elliptic $L$-parameter. For every $\pi \in \Irr_{\Ql}G(F)$ such that $\varphi_{\pi}^{\FS} = \varphi$, the Harish-Chandra character $\Theta_{\pi_0}$ of $\pi_0:=\OS*\pi$ is stable, i.e. $\Theta_{\pi_0}$
		is invariant under $G(\overline{F})$-conjugacy as a non-zero function on the elliptic regular semisimple elements $G(F)_{\elliptic}$ of $G(F)$. Moreover, if $F$ has characteristic zero, then $\Theta_{\pi_0}$ is a non-zero stable distribution on $G(F)$.
	\end{theorem}

	\subsection{Sketch of the proof}

	Let us first recall some relevant results of \cite{fargues2021geometrization}. We refer to Subsection \ref{Subsection_FS} for details. Let $\Bun_G$ be the stack of $G$-bundles over the Fargues--Fontaine curve. Let $\operatorname{D}(\Bun_G)$ be the derived category of sheaves on $\Bun_G$ and $\operatorname{D}(\Bun_G)^{\omega}$ its full subcategory of compact objects. Let $Z^1(W_F, \widehat{G})/\widehat{G}$ be the stack of $L$-parameters. Let $\Perf(Z^1(W_F, \widehat{G})/\widehat{G})$ be the category of perfect complexes on $\ZoneG$. Fargues and Scholze defined an action of $\Perf(\ZoneG)$ on $\operatorname{D}(\Bun_G)^{\omega}$, which is called the \emph{spectral action}.

	Note that we have an open and closed immersion $i_{\varphi}: [*/S_{\varphi}] \to \ZoneG$ (\cite[Section X.2]{fargues2021geometrization}) and an open immersion $i_1: [*/G(F)] \to \Bun_G$. For  a representation $\rho'$ of $S_{\varphi}$ and a finitely generated smooth admissible representation $\pi'$ of $G(F)$, we denote  
	$$\rho' * \pi':=i_1^*\left((i_{\varphi*}\rho') * (i_{1!}\pi')\right) \in \operatorname{D}([*/G(F)], \Ql) \cong \operatorname{D}(\Rep_{\Ql} G(F)),$$
	where the $*$ on the right hand side denotes the spectral action.

		As we have already mentioned,  
			$$\pi_0:=\OS*\pi \in \operatorname{D}(\Rep_{\Ql}G(F))$$ 
			is a derived representation of $G(F)$.
		Moreover, denote $$\mathcal{F}_0:=(i_{\varphi*}\OS) * (i_{1!}\pi).$$ 
		In fact, $\mathcal{F}_0$ is supported on the open substack $[*/G(F)] \subseteq \Bun_G$ (Lemma \ref{Lemma:support}), so 
		$$\pi_0 = i_{1}^*\mathcal{F}_0 \qquad \mathcal{F}_0 \cong i_{1!}\pi_0.$$
		Without further mention, we will always identify $$\operatorname{D}(G(F), \Ql) := \operatorname{D}(\Rep_{\Ql}G(F))=\operatorname{D}([*/G(F)], \Ql)$$ with its image in $\operatorname{D}(\Bun_G)$ via $i_{1, !}$. We will later abuse the notations and do not distinguish between $\pi_0$ and $\mathcal{F}_0$.

	To show that $\Theta_{\pi_0}$ is stable, we observe in Lemma \ref{Prop:Hecke_eigen} that $\mathcal{F}_0$ is a Hecke eigensheaf on $\Bun_G$, which means the following. 
	For any algebraic representation $V$ of $\widehat{G}$, Fargues and Scholze (\cite[Section V]{fargues2021geometrization}) constructed an operator
	$T_V: \operatorname{D}(\Bun_G) \to \operatorname{D}(\Bun_G)$. These are called the \emph{Hecke operators}. A sheaf on $\Bun_G$ is called a \emph{Hecke eigensheaf} if for any algebraic representation $V$ of $\widehat{G}$, 
	$$T_V(\mathcal{F})=\mathcal{F}^{\oplus \dim V}.$$

	Recall that our goal is to show that the Harish-Chandra character $\Theta_{\pi_0}$ of $\pi_0=\OS*\pi$ is stable, where $*$ is the spectral action.
	\begin{enumerate}
		\item Let $g \in G(F)_{\elliptic}$ be an elliptic regular semisimple element.
		\item Let $T_g:=\Cent(g, G)$ be the centralizer of $g$ in $G$, which is an elliptic $F$-torus. We choose a Borel subgroup over $\overline{F}$ containing $(T_g)_{\overline{F}}$. Thus 
		$$X_*(T_g):=\Hom((T_g)_{\overline{F}}, \; \mathbb{G}_{m, \overline{F}})$$ 
		is canonically identified with $X_*(\Tuniv)$ with a distinguished notion of dominance. We will denote by $$X_*:=X_*(\Tuniv) \cong X_*(T_g)$$ 
		for the cocharacters and similarly $X_*^+$ for the dominant cocharacters.
	\end{enumerate}
	For any $g_1 \in G(F)$, let 
	$$[[g_1]]:=\{g_1' \in G(F)\;|\; g_1 \;\text{and}\; g_1' \;\text{are conjugate in} \;G(\overline{F})\}/\{G(F)-\text{conjugacy}\}$$
	be the set of $G(F)$-conjugacy classes inside the stable conjugacy class of $g_1$ in $G(F)$.

	The overall strategy is to express $\Theta_{\pi_{0}}(g)$ as a weighted sum of $\Theta_{\pi_{0}}(g')$ for $g'$ running over a set of representatives for the $G(F)$-conjugacy classes in $[[g]]$, with certain coefficients $c(g')$, i.e.,
	\begin{equation}\label{Eq_average}
		\Theta_{\pi_0}(g)=\sum_{g' \in [[g]]}c(g')\Theta_{\pi_0}(g'),
	\end{equation}
	and then show that $c(g')$ is essentially a constant independent of $g' \in [[g]]$. More precisely, we will have such an expression for each $\mu \in X_*$ with coefficients $c_{\mu}(g')$, and show that $c_{\mu}(g')$ converges to a constant independent of $g'$ as $\mu$ tends to infinity in an appropriate sense.
	
	The first step in our proof of Theorem \ref{Thm_main} is to compute the Harish-Chandra character $\Theta_{T_{V}(\pi_{0})}$ of $T_V(\pi_0):=i_1^*T_V(i_{1!}\pi_0)$ in two ways, for any representation $V$ of $\widehat{G}$, to obtain an expression as Equation (\ref{Eq_average}).
	
	\begin{lemma}\label{Lemma_1}(Equation (\ref{Eq_Hecke_HC}), Lemma \ref{Lemma:*!})
		Let $V_{\mu}$ be the highest weight representation of $\widehat{G}$ with highest weight $\mu \in X^*(\widehat{T})$. Then we have
		\begin{equation}\label{Equation:key}
			\mathcal{T}_{V_{\mu}}\Theta_{\pi_0}=\Theta_{T_{V_{\mu}^{\vee}}(\pi_{0})}=\dim(V_{\mu})\Theta_{\pi_{0}}
		\end{equation}
		as functions on $G(F)_{\elliptic}$, where $\mathcal{T}_{V_{\mu}}$ will be explained below.
	\end{lemma}
	
	As a quick comment on the proof of Lemma \ref{Lemma_1}, the first equality is obtained by combining Lemma \ref{Lemma:*!} and \cite[Theorem 6.5.2]{hansen2022kottwitz}, which is essentially a corollary of the relative Lefschetz--Verdier trace formula.
	The second equality follows from the Hecke eigensheaf property of $\mathcal{F}_0$.

	By definition (\cite[Definition 3.2.7 and Proposition 6.3.5]{hansen2022kottwitz}), 
	$$\frac{1}{\dim V_{\mu}}\mathcal{T}_{V_\mu}\Theta_{\pi_0}(g)$$
	is the weighted sum of Harish-Chander characters $\Theta_{\pi_0}(g')$ over $g' \in [[g]]$, with coefficients
	\begin{equation}\label{Equation_coeff}
		(-1)^d\frac{\sum_{\lambda \in X_*, \; \overline{\lambda}=\inv(g, g')}\dim(V_{\mu}[\lambda])}{\dim(V_{\mu})},
	\end{equation}
	where 
	\begin{enumerate}
		\item $d=\langle \mu, 2\rho_{G}\rangle$, where $\rho_{G}$ is the half sum of positive roots.
		\item $V_{\mu}[\lambda]$ denotes the $\lambda$-weight space of $V_{\mu}$.
		\item $\inv(g, g')$ is a certain element in the Kottwitz set $B(T_g) \cong X_*(T_g)_{\Gamma}$ (Definition \ref{Defition_inv}).
	\end{enumerate}
	In other words, Lemma \ref{Lemma_1} states that
	\begin{equation}\label{Equation_explicit}
		\Theta_{\pi_0}(g)=(-1)^d\sum_{g' \in [[g]]}\frac{\sum_{\lambda \in X_*,\;\overline{\lambda}=\inv(g, g')}\dim V_{\mu}[\lambda]}{\dim V_{\mu}}\Theta_{\pi_0}(g').
	\end{equation}

	The second step in the proof of Theorem \ref{Thm_main} is to show that the coefficients (\ref{Equation_coeff}) are independent of $g' \in [[g]]$ when $\mu$ tends to infinity in an appropriate sense. Indeed, to show $\Theta_{\pi_{0}}$ is stable, it suffices to show the following.
	
	\begin{theorem}(Theorem \ref{Thm_indep})
		For $m \in \mathbb{Z}_{\geq 1}$, let $\mu_m=4m\rho_{G}$.
		For any $g' \in G(F)$ that is conjugate to $g$ in $G(\overline{F})$, the limit
		$$\lim_{m \to \infty}\frac{\sum_{\lambda \in X_*, \;\overline{\lambda}=\inv(g, g')}\dim (V_{\mu_m}[\lambda])}{\dim V_{\mu_m}}$$
		exists and is independent of $g'$.
	\end{theorem}
	
	\begin{remark}
		In the definition of $\mu_m$, the scalar $4$ ensures that  $(-1)^d=1$.  The factor $\rho_{G}$ is to have a nice formula for the character of $V_{\mu_m}$ (Lemma \ref{Lemma_char}).
	\end{remark}
	
	Let $H_g:=\ker(X_*(T_g)_{\Gamma} \to \pi_1(G)_{\Gamma})$. We define $H_g$ in such a way to take care of different strata of $\Bun_G$ for general $G$. $H_g$ is a finite abelian group since $T_g$ is elliptic (Lemma \ref{Lemma:fab}). Note that under our simplifying Assumption \ref{Assumption_intro} in this introduction, $\pi_1(G)_{\Gamma}$ is trivial, and $H_g=X_*(T_g)_{\Gamma}$. 
	
	Now the idea is to use Fourier analysis on the finite abelian group $H_g$, namely, we apply characters $\chi: H_g \to \mathbb{C}^*$ to $$\character V_{\mu_m}=\sum_{\lambda \in X^*(\widehat{T})}\dim V_{\mu_m}[\lambda]e^{\lambda},$$
	see the discussion before Proposition \ref{Prop:dim}. When $\chi$ is the trivial character, 
	$$\chi(\character V_{\mu_m})=\dim V_{\mu_m}.$$ When $\chi$ is nontrivial, denoting $\overline{\lambda}$ the image of $\lambda$ in $H_g$, we have
	$$\chi(\character V_{\mu_m})=\sum_{\lambda} (\dim V_{\mu_m}[\lambda])\chi(\overline{\lambda})=\dim V_{\mu_m}\sum_{h \in H_g}\chi(h)S_{h, m}$$ 
	is a weighted sum of $S_{h, m}$, where
	$$S_{h, m}:=\frac{\sum_{\lambda \in X_*, \; \overline{\lambda}=h \in H_g} \dim (V_{\mu_m}[\lambda])}{\dim V_{\mu_m}},$$
	see Equation (\ref{Eq:rearrange}).
	Using the Weyl character formula, we have the following estimate on the growth of $\chi(\character V_{\mu_m})$ with respect to $m$.
	
	\begin{proposition}
		\label{Prop_estimate}(Proposition \ref{Prop:dim}) Let $k$ be the number of positive roots.
		\begin{enumerate}
			\item $\dim(V_{\mu_m})$ is a polynomial in $m$ of degree $k$.
			\item Assume that $k \geq 1$. For any nontrivial character $\chi$ of $H_g$, 
			$$\chi(\character V_{\mu_m})$$
			is bounded by a polynomial in $m$ of degree less or equal to $k-1$.
		\end{enumerate}
	\end{proposition}

	Therefore, we have
	$$\frac{\chi(\character V_{\mu_m})}{\dim V_{\mu_m}} \to 0 \qquad  m \to \infty.$$ Hence
	$$S_{\chi, m}:=\sum_{h \in H_g}\chi(h)S_{h, m} \to 0 \qquad m \to \infty.$$
	When $\chi$ runs over the non-trivial characters of $H_g$, the coefficients of $S_{h, m} (h \in H_g)$ in $S_{\chi, m}$, i.e., the vectors $(\chi(h))_{h \in H_g}$ span the subspace of $\mathbb{C}^{|H_g|}$ with sum $0$. Therefore,
	$$S_{h, m}-S_{h', m} \to 0 \qquad m \to \infty$$
	for any $h, h' \in H_g$, as the coefficients of $S_{h, m}-S_{h', m}$ has sum $0$. By noticing that for any $g' \in [[g]]$, $\inv(g, g') \in H_g$ (Lemma \ref{Lemma:ker}), we conclude the proof of the first statement of Theorem \ref{Thm_main}. The second statement of Theorem \ref{Thm_main} follows from \cite[Theorem 6.1]{arthur1996local}.
	
	\subsection{Comments on Theorem \ref{Thm_main}}
	
	We should first emphasize that a prior, we prove different things from the classical results, because we don't know the compatibility between Fargues--Scholze's construction and the classical constructions of the local Langlands correspondence in full generality. 
	
	Secondly, to our knowledge, all classical results use heavily the theory of endoscopy, while our proof is independent of the theory of endoscopy, except the deduction of Theorem \ref{Thm_main_distribution} (i.e., the second statement of Theorem \ref{Thm_main}) from Theorem \ref{Thm_main_general} (i.e., the first statement of Theorem \ref{Thm_main}).
	
	Moreover, while the classical theory of endoscopy is not fully developed in positive characteristic, our method also works in positive characteristic. 
	
	Recently, Bezrukavnikov and Varshavsky (\cite{bezrukavnikov2021affine}) proved stability for depth-zero $L$-packets using the geometry of affine Springer fiber in the positive characteristic case. It is not clear but will be interesting to see if there is any connections between their work and our work.

	\subsection{Acknowledgement}
	
	It is a pleasure to thank Peter Scholze for giving me this project, for his idea of considering the Hecke eigensheaf property and invoking the result of \cite{hansen2022kottwitz}, and for suggesting me use Fourier analysis on finite abelian groups to prove the equi-distribution property. I thank my advisors Jessica Fintzen and Peter Scholze  for their interest and suggestions regarding this topic. Moreover, I thank Alexander Bertoloni-Meli, Mikhail Borovoi, Tasho Kaletha, Linus Hamann, David Hansen, Wenwei Li, Xier Ren, Sandeep Varma, Yakov Varshavsky, Haining Wang, and Xiaoxiang Zhou for discussions and comments on this paper. Additionally I would like to thank all members of the arithmetic geometry and representation theory group in Bonn for their constant support.
	
	I am funded by the Hausdorff Center of Mathematics and jointly hosted by the International Max-Planck Research School on Moduli Spaces. I thank both institutes for providing a nice working environment. I am partly supported by DFG via the Leibniz prize of Peter Scholze.

	\section{Basic notions}
	Let $F$ be a non-archimedean local field with residue characteristic $p$. Fix an algebraic closure $\overline{F}$ of $F$. All our algebraic extensions of $F$ are chosen to be inside $\overline{F}$.  Let $F^s$ be the separable closure of $F$ in $\overline{F}$.
	Let $\Gamma:=\Gal(F^s/F)$. Let $W_F$ be the Weil group of $F$.
	
	Let $G$ be a connected reductive group over $F$. Let $\ell$ be a prime number which is different from $p$. Fix an algebraic closure $\Ql$ of the $\ell$-adic numbers $\mathbb{Q}_{\ell}$. Let $\Rep_{\Ql}G(F)$ be the category of smooth representations of $G(F)$ over $\Ql$. From now on, all representations of $G(F)$ are assumed to be smooth. Let $\Irr_{\Ql}G(F)$ be the equivalence classes of irreducible representations of $G(F)$.  We fix a Haar measure on $G(F)$ throughout.

	Let $G(F)_{\rs}$ be the set of regular semisimple elements in $G(F)$. Let $G(F)_{\sr} \subseteq G(F)_{\rs}$ be the set of strongly regular semisimple elements in $G(F)$. Let $G(F)_{\elliptic} \subseteq G(F)_{\sr}$ be the subset of elliptic elements, i.e., those whose centralizer is an elliptic maximal torus.
	
	Let $\Tuniv$ be the universal Cartan of $G$. Let $X_*:=X_*(\Tuniv)$ be the cocharacters of $\Tuniv$ over $\overline{F}$. $X_*$ has a distinguished notion of dominance given by the universal Borel. Let $X_*^+ \subseteq X_*$ denote the dominant cocharacters.

Let $\widehat{G}$ be the Langlands dual group of $G$ over $\Ql$ .
	It comes canonically with a maximal torus and a Borel $\widehat{T} \subseteq \widehat{B} \subseteq \widehat{G}$ with the following properties. 
	\begin{enumerate}
		\item $\widehat{T}$ is dual to $\Tuniv$ in the sense that there is a canonical isomorphism $X_*(\Tuniv) \cong X^*(\widehat{T})$.
		\item $X_*^+$ corresponds to the dominant characters of $\widehat{T}$ with respect to $\widehat{B}$ under the above canonical isomorphism.
	\end{enumerate}

	Let $\widehat{\Phi}:=\Phi(\widehat{G}, \widehat{T})$ be the roots of $\widehat{G}$ with respect to $\widehat{T}$. Let $\Lambda_{\widehat{\Phi}}:=\mathbb{Z}\widehat{\Phi}$ be the root lattice. Let $\Rep_{\widehat{G}}(\Ql)$ be the category of algebraic representations of $\widehat{G}$ over $\Ql$.

	For $\mu \in X_*(\Tuniv) \cong X^*(\widehat{T})$ dominant integral, we denote 
		$$V_{\mu}:=R^0\operatorname{ind}_{\widehat{B}}^{\widehat{G}}\mu \in \Rep_{\widehat{G}}(\Ql)$$ the irreducible representation of $\widehat{G}$ with highest weight $\mu$. Note that $V_{\mu}$ is denoted by $H^0(\mu)$ in \cite{jantzen2003representation}, and is equal to $L(\mu)$, which is defined as the socle of $H^0(\mu)$ in \cite{jantzen2003representation}. This is because in characteristic zero, all representations are semisimple.
	For $m \in \mathbb{Z}_{\geq 1}$, let $\mu_m:=4m\rho_{G} \in X^*(\widehat{T})$ where $\rho_{G}=\frac{1}{2}\sum_{\alpha \in \widehat{\Phi}^+}\alpha$. 
	
	For $\varphi: W_F \to \widehat{G}(\Ql)$ an $L$-parameter, we denote by $$S_{\varphi}:=\Cent(\varphi, \widehat{G})=\{g \in \widehat{G} \;|\; g\varphi(w)w(g^{-1})=\varphi(w), \forall w \in W_F\}$$ the centralizer of $\varphi$, which is an algebraic group over $\Ql$. We denote by $Z(\widehat{G})$ the center of $\widehat{G}$, which is a group scheme over $\Ql$. We write $\ZG$ for the $\Gamma$-fixed points of $Z(\widehat{G})$, which is again a group scheme over $\Ql$. We denote by
		$$\pi_1(G):=\pi_1(G_{\overline{F}}) \cong X_*/\{{\text{coroot lattice}}\} \cong X^*(\widehat{T})/\Lambda_{\widehat{\Phi}}$$
		the Borovoi fundamental group of $G$ (see Subsection \ref{Subsection_Borovoi} for details). Then $\ZG$ is the diagonalizable group scheme with characters the coinvariance $\pi_1(G)_{\Gamma}$ of $\pi_1(G)$ (Lemma \ref{Lemma:ZG}).

	\subsection{Algebraic groups and the regular representation}\label{Subsection_regular_rep}

	\begin{definition}[{\cite[I.2.10]{jantzen2003representation}}]
		Let $H$ be an algebraic group over $\Ql$. The \textbf{regular representation} of $H$ is defined to be its coordinate ring $\Ql[H]=\mathcal{O}(H)$ as a left-$H$-module.
	\end{definition}
	
	Let $*:=\Spec \Ql$. Let $[*/H]$ be the classifying stack of $H$. Then the category of quasicoherent sheaves on $[*/H]$ is equivalent to the category of algebraic representations of $H$. Under this equivalence, the regular representation corresponds to $(\pr)_*\Ql$, where $\pr: * \to [*/H]$ is the projection map, and $\Ql$ is the structure sheaf on $*$.

	\begin{lemma}\label{Lemma_O(S_phi)}
		$$\mathcal{O}(S_{\varphi}) \cong \bigoplus_WW^{\oplus\dim W},$$
		where $W$ runs over all irreducible representations of $S_{\varphi}$.
	\end{lemma}
	
	\begin{proof}
		Over $\Ql$, the regular representation $\mathcal{O}(S_{\varphi})=\Ql[S_{\varphi}]$ is semisimple. It remains to compute the multiplicity for each irreducible representation $W$, which turns out to be $\dim W$ by Frobenius reciprocity.
	\end{proof}
	
	We can decompose the regular representation $\mathcal{O}(S_{\varphi})$ according to its restriction to the central subgroup $Z(\widehat{G})^{\Gamma}$: \begin{equation}\label{Eq:res_to_Z}
		\mathcal{O}(S_{\varphi}) \cong \bigoplus_{\chi \in X^*(Z(\widehat{G})^{\Gamma})}\mathcal{O}(S_{\varphi})_{\chi},
	\end{equation}
	where $\mathcal{O}(S_{\varphi})_{\chi}$ is the $\chi$-isotypic summand of $\mathcal{O}(S_{\varphi})$. In particular, $\mathcal{O}(S_{\varphi})_0$ is the $\chi$-isotypic summand for $\chi=0$.
	
	\begin{lemma}\label{Lemma_OSZ}
		Let $\OSZ$ be the regular representation of $S_{\varphi}/\ZG$, viewed as a representation of $S_{\varphi}$ by inflation. We have
		$$\Oo \cong \mathcal{O}\left(S_{\varphi}/\ZG\right).$$
	\end{lemma}
	
	\begin{proof}
			Indeed,
			$$\mathcal{O}\left(S_{\varphi}/\ZG\right) \cong \mathcal{O}(S_{\varphi})^{\ZG} \cong \Oo,$$
			where the second isomorphism follows from Equation \eqref{Eq:res_to_Z}.
	\end{proof}

	\subsection{Borovoi's algebraic fundamental group} \label{Subsection_Borovoi}
	
	In this subsection, we review the theory of Borovoi's algebraic fundamental group. The main references are \cite{borovoi1998abelian},  \cite{borovoi2014algebraic}, and \cite{borovoi2024mo}.
   
    Let $G$ be a connected reductive group over a field $K$. Let $\Gamma_K:=\Gal(K^s/K)$ denote the absolute Galois group of $K$.
    
    We write $G_{\operatorname{der}}=[G, G]$ for the derived group of $G$ (it is semisimple). We denote by $G_{\operatorname{sc}}$ the universal cover of $G_{\operatorname{der}}$ (it is simply connected) and consider the composite homomorphism
    $$ \rho\colon\, G_{\operatorname{sc}}\twoheadrightarrow G_{\operatorname{der}}\hookrightarrow G.$$
    
    Let $T \subseteq G$ be a maximal torus (defined over $K$). 
    Set $T_{\operatorname{sc}}=\rho^{-1}(T) \subseteq G_{\operatorname{sc}}$. Let $X_*(T)$ denote the cocharacter group of $T$ (over $K^s$). We set
    $$\pi_1(G, T):=X_*(T)/\rho_*X_*(T_{\operatorname{sc}}).$$
    The Galois group $\Gamma_K$ naturally acts on $\pi_1(G, T)$.
    
    \begin{proposition}[{\cite[Lemma 1.2]{borovoi1998abelian}}]
    	For any two maximal tori $T_1,T_2\subseteq G$,
    	there is a canonical isomorphism of $\Gamma_K$-modules
    	$$\varphi_{12}\colon\, \pi_1(G,T_1)\overset\sim\longrightarrow \pi_1(G,T_2).$$
    	Moreover, for any third maximal  torus $T_3\subseteq G$, we have
    	$$\varphi_{13}=\varphi_{23}\circ\varphi_{12}$$
    	with the obvious notations.
    \end{proposition}
    
    \begin{proof}
    	Let $T_1,T_2\subseteq G$ be two maximal tori.
    	Then there exists an element $g\in G(K^s)$ such that
    	\begin{equation}\label{e:2}
    		T_2=g \cdot T_1\cdot g^{-1}.
    	\end{equation}
    	We obtain an isomorphism
    	$$g_*\colon \pi_1(G,T_1)\overset\sim\longrightarrow \pi_1(G,T_2).$$
    	
    	By Lemma \ref{Lemma_g} below, $g_*$ is independent of the choice of $g \in G(K^s)$. By Lemma \ref{Lemma_Gamma} below, $g_*$ is $\Gamma_K$-equivariant. The last claim follows since $g_*$ is independent of the choice of $g \in G(K^s)$.
    \end{proof}
    
    \begin{lemma}\label{Lemma_g}
    	The isomorphism $g_*$ above does not depend on the choice of $g$ satisfying \eqref{e:2}.
    \end{lemma}
    
    \begin{proof}
    	Let $g'\in G(K^s)$ be another element satisfying  \eqref{e:2}.
    	Then
    	$$ g^{-1}g'\cdot T_1 \cdot  ( g^{-1}g')^{-1} =  T_1.$$
    	Let $N_1$ denote the normalizer of $T_1$ in $G$.
    	Set $n=g^{-1} g'$.
    	Then $n\in N_1(K^s)$ and  $g'=gn$,
    	whence
    	$$ g'_*=g_*\circ n_*.$$
    	By Lemma \ref{Lemma_N} below, the group $N_1(K^s)$, when acting on $X_*(T_1)$ and $X_*((T_1)_{\operatorname{sc}})$ by conjugation,
    	acts trivially on $\pi_1(G,T_1)$, which completes the proof of the lemma.
    \end{proof}
    
    \begin{lemma}[{\cite{borovoi2024mo}}]\label{Lemma_Gamma}
    	The isomorphism $g_*$ above preserves the action of the Galois group $\Gamma_K={\rm Gal}(K^s/K)$.
    \end{lemma}
    
    \begin{proof}
    	Let
    	$$x_1\in\pi_1(G,T_1),\quad  x_2= g_*(x_1)\in\pi_1(G,T_2),\quad  \gamma\in\Gamma_K.$$
    	Then
    	$$^\gamma\!x_2=\,^\gamma\! g_*(\,^\gamma\! x_1)$$
    	where by abuse of notation we write $^\gamma\! g_*$ for $\,(^\gamma\! g)_*$.
    	We obtain from \eqref{e:2} that
    	$$^\gamma T_2={}^\gamma\!g \cdot {}^\gamma T_1\cdot {}^\gamma\!g^{-1}. $$
    	Since $T_1$ and $T_2$ are defined over $K$, we have $^\gamma T_1=T_1$ and $^\gamma T_2=T_2$.
    	Therefore, we obtain that
    	$$ T_2={}^\gamma\!g \cdot  T_1\cdot {}^\gamma\!g^{-1}. $$
    	Comparing with \eqref{e:2}, we see from Lemma \ref{Lemma_g} that $({}^\gamma\! g)_*=g_*$, that is,
    	$$^\gamma\!x_2=g_*(\,{}^\gamma\! x_1).$$
    	Thus our isomorphism $g_*$ preserves the $\Gamma_K$-action, as desired.
    \end{proof}
    
    \begin{lemma}\label{Lemma_N}
    	Let $T\subseteq G$ be a  maximal torus of a reductive group over a field $K$.
    	Write $N$ for the normalizer of $T$ in $G$.
    	Then $N(K^s)$, when acting  on $X_*(T)$ and $X_*(T_{\operatorname{sc}})$, acts on $\pi_1(G,T)$ trivially.
    \end{lemma}
    
    \begin{proof}
    	The group $N(K^s)$ acts on $X_*(T)$, $X_*(T_{\operatorname{sc}})$, and $\pi_1(G,T)$ via the Weyl group $W=W(G,T)=N(K^s)/T(K^s)$.
    	
    	Let $X^*(T)$ denote the character group of $T$ (over $K^s$).
    	Let $R=R(G,T)\subset X^*(T)$ and $R^\vee=R^\vee(G,T)\subset X_*(T)$ denote the corresponding root and coroot systems.
    	Then $W$ is generated by the reflections $s_\alpha$ for $\alpha\in R$.
    	It suffices to show that each $s_\alpha$ acts trivially on $\pi_1(G,T)$.
    	
    	Let $\alpha^\vee\in R^\vee$ denote the coroot corresponding to the root $\alpha$.
    	Then the reflection $s_\alpha$ acts on $X_*(T)$ by
    	$$ s_\alpha(u)= u-\langle\alpha,u\rangle \alpha^\vee $$
    	for $u\in X_*(T)$;
    	see \cite[Section 1.1]{springer1977reductive}.
    	Since $\alpha^\vee\in R^\vee\subseteq \rho_*X_*(T_{\operatorname{sc}})$,
    	we see that $s_\alpha$ indeed acts trivially on $\pi_1(G,T)=X_*(T)/\rho_*X_*(T_{\operatorname{sc}})$.
    	
    \end{proof}
    
		In our context, It follows that for a connected reductive group $G$ over $F$, we have $\Gamma=\Gal(F^s/F)$-equivariant identifications
		\begin{equation}\label{Eq_Borovoi}
			\pi_1(G) \cong X_*(T)/X_*(T_{\operatorname{sc}}) \cong X_*/\{\text{coroot lattice}\} \cong X^*(\widehat{T})/\Lambda_{\widehat{\Phi}}
		\end{equation}
		for any maximal torus $T$ of $G$.

	\subsection{The Kottwitz set}\label{Subsection_Kottwitz}
	
	In this subsection, we review the theory of the Kottwitz set. The main reference is \cite{kottwitz1985isocrystal}. For a more modern approach, see \cite{kottwitz2014bg}.
	
	Recall that the Kottwitz set $B(G)$ is the set of $\sigma$-conjugacy classes in $G(\breve{F})$. For any element $b \in G(\breve{F})$, we can associate a homomorphism $\nu: \mathbb{D} \to G$, where $\mathbb{D}$ is the diagonalizable group over $F$ with character group $\mathbb{Q}$. We say that an element $b \in G(\breve{F})$ is \emph{basic} if the corresponding homomorphism $\nu: \mathbb{D} \to G$ factors through the center of $G$. 
	Let $B(G)_{\basic}$ denote the set of basic elements in $B(G)$.
	
	\begin{theorem}[{\cite[2.4]{kottwitz1985isocrystal}}]\label{Thm_Kottwitz_tori}
		Let $T$ be a $F$-torus. There is a functorial (in $T$) isomorphism
		$$X_*(T)_{\Gamma} \cong B(T).$$
	\end{theorem}

	\begin{theorem}[{\cite[Proposition 5.3]{kottwitz1985isocrystal}}]\label{Thm_factor_basic}
		Let $T$ be an elliptic maximal $F$-torus of $G$. The image of the natural map 
		$B(T) \to B(G)$
		is $B(G)_{\basic}$.
	\end{theorem}

	\begin{theorem}[{\cite[Proposition 5.6]{kottwitz1985isocrystal}}]
		There is a unique functorial isomorphism
		$$B(G)_{\basic} \cong X^*(\ZG)$$
		that agrees with Theorem \ref{Thm_Kottwitz_tori} for tori.
	\end{theorem}
	
	\begin{lemma}\label{Lemma:ZG}
		$\pi_1(G)_{\Gamma} \cong X^*(\ZG)$.
	\end{lemma}
	
	\begin{proof}
			This follows from the short exact sequence of diagonalizable group schemes
			$$1 \to Z(\widehat{G}) \to \widehat{T} \to \widehat{T}/Z(\widehat{G}) \to 1.$$
	\end{proof}

	\subsection{Recollections on Fargues--Scholze}\label{Subsection_FS}
	
	In \cite{fargues2021geometrization}, Fargues--Scholze formulated the local Langlands correspondence as the geometric Langlands correspondence over the Fargues--Fontaine curve. We briefly summarize the results that we will use. We refer the readers to \cite[Introduction]{fargues2021geometrization} for details.

	Let $\Bun_G$ be the moduli stack of $G$-bundles over the Fargues--Fontaine curve. The stack $\Bun_G$ admits a stratification by locally closed substacks
	$$i_b: \Bun_G^b \subseteq \Bun_G$$
	for $b \in B(G)$ (\cite[Theorem III.0.2]{fargues2021geometrization}). For $b \in B(G)$, $\Bun_G^b$ is isomorphic to the classifying stack $[*/\tilde{G_b}]$, where $\tilde{G_b}$ is an extension of $G_b(F)$ by a connected unipotent group. Fargues and Scholze define in \cite[Chp.~VII]{fargues2021geometrization} a derived category of $\ell$-adic sheaves
	$$\operatorname{D}(\Bun_G, \Ql):=\operatorname{D}_{\operatorname{lis}}(\Bun_G, \Ql)$$
	on $\Bun_G$. $\operatorname{D}(\Bun_G, \Ql)$ admits a semi-orthogonal decomposition into $\operatorname{D}(\Bun_G^b, \Ql)$ with $b \in B(G)$, and
	$$\operatorname{D}(\Bun_G^b, \Ql) \cong \operatorname{D}([*/G_b(F)], \Ql) \cong \operatorname{D}(G_b(F), \Ql):=\operatorname{D}(\Rep_{\Ql}G_b(F))$$
	is equivalent to the derived category of smooth $G_b(F)$-representations.
	
	Inside $\Bun_G$, there is an open locus $\Bun_G^{\operatorname{ss}}$ consisting of semistable $G$-bundles. The points of $\Bun_G^{\operatorname{ss}}$ correspond to basic elements in $B(G)$, and we have an open and closed decomposition
	$$\Bun_G^{\operatorname{ss}} \cong \sqcup_{b \in B(G)_{\basic}}\Bun_G^b.$$
	Such $\Bun_G^b$ corresponding to $b \in B(G)_{\basic}$ is called a \emph{basic stratum}.

	For any representation $V$ of $\widehat{G}$, Fargues and Scholze define in 
	\cite[Chp.~IX]{fargues2021geometrization} a \emph{Hecke operator}
	$$T_V: \operatorname{D}(\Bun_G, \Ql) \to \operatorname{D}(\Bun_G, \Ql)^{BW_F},$$
	where the superscript $BW_F$ means $W_F$-equivariant objects. For most of the time, we forget the $W_F$-equivariance and simply view $T_V$ as an endomorphism of $\operatorname{D}(\Bun_G, \Ql)$.
	
	\begin{lemma}[{\cite[Lemma 5.3.2]{zou2022categorical}}]\label{Lemma_pi_1}
		The Hecke operators $T_V$ are compatible with $\pi_1(G)_{\Gamma}$-gradings.
	\end{lemma}
	
	\begin{remark}\label{Remark:ULA_compact_finite}
		There are two finiteness notions for objects in $\operatorname{D}(\Bun_G, \Ql)$ -- $\ULA$ object and compact object. An object $A \in \operatorname{D}(\Bun_G, \Ql)$ that is both $\ULA$ and compact is called \textbf{finite}. Explicitly, $A$ is finite if and only if the following two properties hold: $i_b^*A \in \operatorname{D}(G_b(F), \Ql)$ is zero for all but finitely many $b \in B(G)$, and $\oplus_nH^n(i_b^*A)$ is a finite length $G_b(F)$-representation for every $b \in B(G)$ (\cite[Definition 1.6.2]{hansen2023beijing}).
	\end{remark}
	
	\begin{lemma}[{\cite[Theorem IX.2.2]{fargues2021geometrization}}]\label{Lem_Hck_ULA}
		The Hecke operators preserve $\ULA$ objects and compact objects.
	\end{lemma}

	Fargues and Scholze define in \cite[Chp.~VIII]{fargues2021geometrization} a stack of $L$-parameters $$Z^1(W_F, \widehat{G})/\widehat{G}.$$ Let 
	$$\mathcal{Z}(G(F), \Ql):=\End(\id_{\operatorname{D}(G(F), \Ql)})$$
	be the Bernstein center of $\operatorname{D}(G(F), \Ql)$. Using the Hecke operators and the formalism of excursion,  Fargues--Scholze construct a map
	$$\Psi_G: \mathcal{O}(Z^1(W_F, \widehat{G})/\widehat{G}) \to \mathcal{Z}(G(F), \Ql).$$
	For any irreducible representation $\pi'$ of $G(F)$, 
	the composition 
	
	\begin{center}
		\begin{tikzcd}
			{\mathcal{O}(Z^1(W_F, \widehat{G})/\widehat{G})} & {\End(\id)} & {\End(\pi')} & \Ql
			\arrow["{\Psi_G}", from=1-1, to=1-2]
			\arrow[from=1-2, to=1-3]
			\arrow["\cong", from=1-3, to=1-4]
		\end{tikzcd}
	\end{center}
	gives rise to a $\Ql$-point of the coarse moduli space of $L$-parameters, i.e., a semisimplified $L$-parameter, which we will denote by $\varphi_{\pi'}^{\FS}$. We call $\varphi_{\pi'}^{\FS}$ the \emph{Fargues--Scholze $L$-parameter} of $\pi'$.
	
	Moreover, the map $\Psi_G$ lifts to a map to 
	$$\mathcal{Z}^{\operatorname{geom}}(G, \Ql):=\End(\id_{\operatorname{D}(\Bun_G, \Ql)})$$
	such that the following diagram commutes (\cite[Corollary IX.0.3]{fargues2021geometrization}):
	
\[\begin{tikzcd}
	{\mathcal{O}(Z^1(W_F, \widehat{G})/\widehat{G})} & {\mathcal{Z}^{\operatorname{geom}}(G, \Ql)} \\
	& {\mathcal{Z}(G(F), \Ql)}
	\arrow["\Psi", from=1-1, to=1-2]
	\arrow["{\Psi_G}"', from=1-1, to=2-2]
	\arrow[from=1-2, to=2-2]
\end{tikzcd}.\]

    \begin{lemma}[{\cite[Theorem IX.0.5 (iii)]{fargues2021geometrization}}]\label{Lemma_central_char}
			The correspondence $\pi' \mapsto \varphi_{\pi'}^{\FS}$ is compatible with central characters.
	\end{lemma}
	
	\begin{lemma}[{\cite[Corollary IX.7.3]{fargues2021geometrization}}]\label{Lemma_parabolic}
		The correspondence $\pi' \mapsto \varphi_{\pi'}^{\operatorname{FS}}$ is compatible with parabolic induction.
	\end{lemma}

	Moreover, Fargues--Scholze constructed in \cite[Chp.~X]{fargues2021geometrization} an action of the perfect complexes  $\Perf(Z^1(W_F, \widehat{G})/\widehat{G})$ on the compact objects $\operatorname{D}(\Bun_G, \Ql)^{\omega}$, called the spectral action. For $B \in \Perf(Z^1(W_F, \widehat{G})/\widehat{G})$ and $A \in \operatorname{D}(\Bun_G, \Ql)^{\omega}$, we will denote by $B*A \in \operatorname{D}(\Bun_G, \Ql)^{\omega}$ the spectral action of $B$ on $A$. 
	
	The construction of the spectral action is by assembling the Hecke operators, thus we have the following lemma.
	
	\begin{lemma}[{\cite[Corollary X.1.3]{fargues2021geometrization}}]\label{Lemma_Spectral_Hecke}
		The spectral action is compatible with the Hecke action. 
	\end{lemma}

	\begin{lemma}\label{Lemma_spectral_ULA}
		The spectral action preserves $\ULA$ objects and compact objects. 
	\end{lemma}
	
	\begin{proof}
		Follows from Lemma \ref{Lem_Hck_ULA} by construction of the spectral action.
	\end{proof}
	
	The spectral action can be made explicit in the case of elliptic $L$-parameters (\cite[Section X.2]{fargues2021geometrization}).
	
	\begin{definition}[{\cite[Definition X.2.1]{fargues2021geometrization}}]\label{Def:elliptic}
		An $L$-parameter $\varphi: W_F \to \widehat{G}(\Ql)$ is \textbf{elliptic} if it is semisimple and the centralizer $S_{\varphi} \subseteq \widehat{G}$ has the property that $S_{\varphi}/Z(\widehat{G})^{\Gamma}$ is finite.
	\end{definition}

	Let $\varphi$ be an elliptic $L$-parameter, which induces a map
	$$i_{\varphi}: */S_{\varphi} \to Z^1(W_F, \widehat{G})/\widehat{G}.$$ According to \cite[Section X.2]{fargues2021geometrization}, the unramified twists of $\varphi$ defines a connected component 
	$C_{\varphi} \subseteq Z^1(W_F, \widehat{G})/\widehat{G}.$ The connected component $C_{\varphi}$ corresponds to an idempotent element of $\mathcal{O}(Z^1(W_F, \widehat{G})/\widehat{G})$, which gives rise to an idempotent element of  $\mathcal{Z}^{\operatorname{geom}}(G(F), \Ql)$ under the ring homomorphism
	$$\Psi: \mathcal{O}(Z^1(W_F, \widehat{G})/\widehat{G}) \to \mathcal{Z}^{\operatorname{geom}}(G, \Ql).$$ Thus, there is a corresponding direct summand
	$$\operatorname{D}^{C_{\varphi}}(\Bun_G, \Ql)^{\omega} \subseteq \operatorname{D}(\Bun_G, \Ql)^{\omega},$$
	consisting of the objects whose irreducible subquotients have Fargues--Scholze $L$-parameter an unramified twist of $\varphi$. Since $C_{\varphi}$ is regular, pushforward along the inclusion $*/S_{\varphi} \to C_{\varphi}$ preserves perfect complex. Hence the spectral action induces an action of $\Perf(*/S_{\varphi})$ on $\operatorname{D}(\Bun_G, \Ql)^{\omega}$. Noting that the spectral action is compatible with the map $\Psi$ between categorical centers, it follows that the spectral action, hence the induced action of $\Perf(*/S_{\varphi})$, preserves $\operatorname{D}^{C_{\varphi}}(\Bun_G, \Ql)^{\omega}$.
	
	\begin{lemma}[{\cite[Section X.2]{fargues2021geometrization}}]\label{Lemma:support}Let $A \in \operatorname{D}^{C_{\varphi}}(\Bun_G, \Ql)^{\omega}$, then $A$ is supported on $\Bun_G^{\operatorname{ss}}$.
	\end{lemma}
	
	\begin{proof}
		This follows from $\varphi$ being elliptic and Lemma \ref{Lemma_parabolic}.
	\end{proof}
	
	\begin{lemma}\label{Lemma:*!}
		Let $A \in \operatorname{D}^{C_{\varphi}}(\Bun_G, \Ql)^{\omega}$, then $i_{1, *}A \cong i_{1, !}A.$
	\end{lemma}
	
	\begin{proof}
		 We note that $i_{1, *}A \in \operatorname{D}^{C_{\varphi}}(\Bun_G, \Ql)^{\omega}$ because of the following commutative diagram.
		\[\begin{tikzcd}
			{\mathcal{O}(Z^1(W_F, \widehat{G})/\widehat{G})} & {\mathcal{Z}^{\operatorname{geom}}(G, \Ql)} \\
			& {\mathcal{Z}(G(F), \Ql)}
			\arrow[from=1-1, to=1-2]
			\arrow[from=1-1, to=2-2]
			\arrow[from=1-2, to=2-2]
		\end{tikzcd}\]
		Hence the lemma follows from Lemma \ref{Lemma:support}. 
		
	\end{proof}

	Let $\varphi$ be an elliptic $L$-parameter. Let $\pi$ be an irreducible representation of $G(F)$ with Fargues--Scholze $L$-parameter $\varphi$.
	Define 
	\begin{equation}\label{Eq_def_pi}
		\mathcal{F}_0:=(i_{\varphi})_*\OSZ * (i_1)_!\pi  \qquad \pi_0:=i_1^*\mathcal{F}_0.
	\end{equation}

	Since $(i_{\varphi})_*\OSZ \in \Perf(Z^1(W_F, \widehat{G})/\widehat{G})$ and $(i_1)_!\pi$ is a finite object (see Remark \ref{Remark:ULA_compact_finite}) of $\operatorname{D}(\Bun_G, \Ql)$, the sheaf $\mathcal{F}_0$ is also a finite object of $\operatorname{D}(\Bun_G, \Ql)$ (Lemma \ref{Lem_Hck_ULA}). 
	Recall that the spectral action preserves $\operatorname{D}^{C_{\varphi}}(\Bun_G, \Ql)^{\omega}$,  we have $\mathcal{F}_0 \in \operatorname{D}^{C_{\varphi}}(\Bun_G, \Ql)^{\omega}$ since $(i_1)_!\pi \in \operatorname{D}^{C_{\varphi}}(\Bun_G, \Ql)^{\omega}$. Hence $\mathcal{F}_0$ is supported on the basic strata (Lemma \ref{Lemma:support}).
	Therefore, 
	$$\mathcal{F}_0 \in \operatorname{D}(\Bun_G^{\operatorname{ss}}, \Ql) \cong \bigoplus_{b \in B(G)_{\basic}}\operatorname{D}(G_b(F), \Ql).$$

	\begin{proposition}\label{Proposition_elliptic}
		$\mathcal{F}_0$ is a finite direct sum of irreducible representations of $G_b(F)$ for $b \in B(G)_{\basic}$ up to degree shifts. Moreover, all the irreducible representations occurring in $\mathcal{F}_0$ are supercuspidal with Fargues-Scholze $L$-parameter $\varphi$ and have the same central character.
	\end{proposition}
	
	\begin{proof}
		By the compatibility of $\Psi_G$ with parabolic induction, any irreducible subquotient of $\mathcal{F}_0$ is supercuspidal. By \cite[Theorem 5.4.1]{casselman1995introduction}, those supercuspidal irreducible representations have no non-trivial extensions between each other, as all of them have the same central character by Lemma \ref{Lemma_central_char}.
		Therefore, $\mathcal{F}_0$ is completely reducible.
		Hence, we can write
		\begin{equation}\label{Eq:decomp}
			\mathcal{F}_0 \cong \bigoplus_{\pi'}\pi'[d_{\pi'}],
		\end{equation}
		where $\pi'$ runs over irreducible subquotients of $\mathcal{F}_0$ (here different $\pi'$ might be isomorphic as representation) and $d_{\pi'} \in \mathbb{Z}$ is the degree shift. Since $\mathcal{F}_0$ is a finite object, there are only finitely many summands occurring in $\mathcal{F}_0$.
	\end{proof}

	\begin{remark}
		\begin{enumerate}
			\item For now, as in Equation (\ref{Eq:decomp}), we only know that $\mathcal{F}_0$ is a direct sum of irreducible representations \textbf{up to degree shifts}. We expect that these are honest representations, i.e., $d_{\pi'}=0$.
			\item We expect that the set of $\pi'$'s occuring in Equation \eqref{Eq:decomp} is exactly the Fargues--Scholze $L$-packet  of $\varphi$ (see \eqref{Eq_FS_L-pacekt}). However, this is in general unknown so far.
		\end{enumerate}
	\end{remark}

	\subsection{Harish-Chandra characters}

	In this subsection, we review the theory of Harish-Chandra characters. The main reference is \cite{fintzensupercuspidal},  \cite{kalethalocal}, and \cite{taibijacquet}.

	Let $(V, \pi)$ be an admissible representation of $G(F)$. Any $f \in C_c^{\infty}(G(F))$ gives an endomorphism of $V$
	$$\pi(f): v \mapsto \int_{G(F)} f(g)\pi(g)v \,dg,$$
	and its image is contained in $V^K$ for some compact open subgroup $K$ of $G(F)$. In particular, $\pi(f)$ has finite range and we may consider 
	$$\Theta_{\pi}(f):=\tr\pi(f).$$
	A standard result in representation theory of finite-dimensional associative algebras implies that the Harish-Chandra characters $\Theta_\pi$ of the irreducible smooth representations of $G(F)$ (up to isomorphism) are linearly independent. In particular a smooth representation of finite length is characterized up to semisimplification by its Harish-Chandra character.
	
	Denote by $G_{\rs}$ the regular semisimple locus in $G$, an open dense subscheme.
	
	\begin{theorem}[{\cite[Theorem 16.3]{harishchandra1999admissible}}]
		Assume that $F$ is a non-archimedean local field of characteristic zero. Let $(V, \pi)$ be an irreducible smooth representation of $G(F)$. Choose a Haar measure of $G(F)$. There exists a unique element in $L^1_{\operatorname{loc}}(G(F))$, also denoted $\Theta_\pi$, such that for any $f \in C_c^{\infty}(G(F))$ we have
		$$\tr\pi(f)=\int_{G(F)}\Theta_{\pi}(g)f(g) \,dg.$$
		Moreover, $\Theta_{\pi}$ is represented by a unique locally constant function on $G_{\rs}(F)$.
	\end{theorem}

			\begin{remark}
				A similar result holds for $F$ of positive characteristic, see \cite[Section 5]{fintzensupercuspidal}.
			\end{remark}
			
			\begin{definition}\label{Def_HC}
				We define the \textbf{Harish-Chandra character} of a derived representation $A \in \operatorname{D}(G(F), \Ql)$ of finite length as the alternating sum of the Harish-Chandra characters of its cohomologies.
			\end{definition}

			\subsection{Recollections on Hansen--Kaletha--Weinstein}
			
			In this subsection, we review \cite{hansen2022kottwitz}.
			
			For $b \in G(\breve{F})$, let $\pi_b \in \Rep_{\Ql}(G_b(F))$ be a finite length and admissible representation. For $\mu: \mathbb{G}_m \to G$ a conjugacy class of cocharacters defined over $\overline{F}$, let $R\Gamma(G, b, \mu)[\pi_b]$ be the $\pi_b$-isotypic component of the cohomology of the tower of shtuka spaces $\operatorname{Sht}_{G, b, \mu}$ (\cite[Section 1]{hansen2022kottwitz}). $R\Gamma(G, b, \mu)[\pi_b]$ is a bounded complex whose cohomologies are finite-length admissible representations of $G(F)$. Therefore, we can form the finite-length virtual representation 
			$$\Mant_{G, b, \mu}(\pi_b)=\sum_i(-1)^iH^i(R\Gamma(G, b, \mu)[\pi_b]),$$ 
			which then has a Harish-Chandra character $\Theta_{\Mant_{G, b, \mu}(\pi_b)}$.
			
			The main result we will use from \cite{hansen2022kottwitz} is the following.
			
			\begin{theorem}[{\cite[Theorem 6.5.2]{hansen2022kottwitz}}]\label{Theorem_HKW}
				We have an equality
				\begin{equation}\label{Equation_HKW}
					\Theta_{\Mant_{G, b, \mu}(\pi_b)}(g)=(T_{b, \mu}^{G_b \to G}\Theta_{\pi_b})(g):=(-1)^d\sum_{(g, g', \lambda) \in \Rel_b}\dim V_{\mu}[\lambda]\Theta_{\pi_b}(g')
				\end{equation}
				for any $g \in G(F)_{\elliptic}$, where $d=\langle \mu, 2\rho_G\rangle \in \mathbb{Z}$ and $\Rel_b$ is certain subset of $G(F) \times G_b(F) \times X^*(\widehat{T})$ (\cite[Definition 3.2.4]{hansen2022kottwitz}).
			\end{theorem}
			
			Let us comment on the proof of Theorem \ref{Theorem_HKW}. Modulo out some technical steps, Theorem \ref{Theorem_HKW} is a corollary of the relative Lefschetz--Verdier trace formula, a modern enhancement of the Lefschetz fixed point formula under the framework of the $2$-category of cohomological correspondences.

			\section{Polynomial growth of $\chi(\character V_{\mu_m})$}
			Recall that $\mu_m=4m\rho_{G}$. The goal of this section is to show that the dimension of the highest weight representation $\dim V_{\mu_m}$ is a polynomial in $m \in \mathbb{Z}_{\geq 0}$ of degree $|\widehat{\Phi}^+|$ and that $\chi(\character V_{\mu_m})$ is bounded by a polynomial of degree less or equal to $|\widehat{\Phi}^+|-1$ for any nontrivial complex character $\chi$ of certain finite abelian group $H_g$ (assuming $|\widehat{\Phi}| \geq 1$).
			Throughout the section, the coefficients of representation is an algebraically closed field $C$ of characteristic $0$, for example, $\Ql$ or $\mathbb{C}$. 
			
			\subsection{Lemmas from Lie theory}
			
			Let us first define the character of an algebraic representation of $\widehat{G}$.
			
			\begin{definition}
				Let $V \in \Rep_{\widehat{G}}(C)$, the \textbf{character} of $V$ is 
				$$\character V := \sum_{\mu 
					\in X^*(\widehat{T})} \dim V[\mu]e^\mu \quad \in \quad \mathbb{Z}[X^*(\widehat{T})] \cong \mathbb{Z}[X_*(\Tuniv)],$$
				where $V[\mu]$ is the $\mu$-weight space of $V$ and $\dim V[\mu]$ is its weight multiplicity.
			\end{definition}
			
			For $m \in \mathbb{Z}_{\geq 0}$, let $\mu_m:=4m\rho_G \in \Lambda_{\widehat{\Phi}} = \mathbb{Z}\widehat{\Phi}$, where $\rho_G:=\frac{1}{2}\sum_{\alpha \in \widehat{\Phi}^+}\alpha$ is the half sum of positive roots.
			
			\begin{lemma}\label{Lemma_char}
				Let $V_{\mu_m}$ be the highest weight representation of $\widehat{G}$ with highest weight  $\mu_m$. We have
				
				$$\character V_{\mu_m}=\prod_{\alpha \in \widehat{\Phi}^+}(e^{2m\alpha}+e^{(2m-1)\alpha}+...+e^{-2m\alpha}).$$

			\end{lemma}

			\begin{proof}
				By \cite[Lemma 24.3]{fulton2004representation}, we have
				$$\sum_{w \in W}(-1)^{l(w)}e^{w\rho_G}=\prod_{\alpha \in \widehat{\Phi}^+}(e^{\alpha/2}-e^{-\alpha/2}).
				$$
				Similarly, for any $c \in \mathbb{Z}_{\geq 1}$, we have
				\begin{equation}\label{Eq:product}
					\sum_{w \in W}(-1)^{l(w)}e^{cw\rho_G}=\prod_{\alpha \in \widehat{\Phi}^+}(e^{c\alpha/2}-e^{-c\alpha/2}).
				\end{equation}
				
				Recall that we have the Weyl character formula
				$$\character(V_{\mu})=\frac{\sum_{w \in W}(-1)^{l(w)}e^{w(\rho_G+\mu)}}{\sum_{w \in W}(-1)^{l(w)}e^{w\rho_G}}=\frac{\sum_{w \in W}(-1)^{l(w)}e^{w(\rho_G+\mu)}}{\prod_{\alpha \in \widehat{\Phi}^+}(e^{\alpha/2}-e^{-\alpha/2})}.$$
				
				Let $\mu=\mu_m$, combining the Weyl character formula and Equation (\ref{Eq:product}), we have

				\begin{align*}    	
					& \character(V_{\mu_m})  \\
					=\;&  \frac{\sum_{w \in W}(-1)^{l(w)}e^{w(\rho_G+4m\rho_G)}}{\prod_{\alpha \in \widehat{\Phi}^+}(e^{\alpha/2}-e^{-\alpha/2})}\\
					=\;& \frac{\prod_{\alpha \in \widehat{\Phi}^+}(e^{(4m+1)\alpha/2}-e^{-(4m+1)\alpha/2})}{\prod_{\alpha \in \widehat{\Phi}^+}(e^{\alpha/2}-e^{-\alpha/2})}
					\\
					=\;& \prod_{\alpha \in \widehat{\Phi}^+}(e^{2m\alpha}+e^{(2m-1)\alpha}+...+e^{-2m\alpha}).
				\end{align*}
			\end{proof}

			Note that for a representation $V$ of $\widehat{G}$ with character
			$$\character V=\sum_{\mu \in X^*(\widehat{T})}\dim V[\mu]e^{\mu},$$
			its dual representation $V^{\vee}$ has character
			$$\character V^{\vee}=\sum_{\mu \in X^*(\widehat{T})}\dim V[\mu]e^{-\mu}$$
			(\cite[23.38]{fulton2004representation}).
			
			\begin{lemma}\label{Lemma_selfdual}
				$V_{\mu_m}$ is self-dual: $V_{\mu_m}^{\vee} \cong V_{\mu_m}$. 
			\end{lemma}
			
			\begin{proof}
				Note that a finite dimensional algebraic representation is uniquely determined by its character. The claim follows since the formula for $\character V_{\mu_m}$ in Lemma \ref{Lemma_char} is invariant under the map $\alpha \mapsto -\alpha, \; \forall \alpha \in \widehat{\Phi}$.
			\end{proof}

			\subsection{Lemmas about $H_g$}
			
			Let $g \in G(F)_{\elliptic}$ and $T_g:=\Cent(g, G)$ with a choice of Borel subgroup over $\overline{F}$ containing $(T_g)_{\overline{F}}$. Recall that by definition of $\pi_1(G)$, we have a $\Gamma$-equivariant short exact sequence (Lemma \ref{Lemma_Gamma} and Equation \eqref{Eq_Borovoi})
			\begin{equation}\label{Equation_SES}
				0  \to X_*((T_g)_{\operatorname{sc}}) \to X_*(T_g) \to \pi_1(G) \to 0,
			\end{equation}
			where $(T_g)_{\operatorname{sc}}$ is the preimage of $T_g$ in $G_{\operatorname{sc}}$, the simply-connected cover of the derived subgroup $G_{\operatorname{der}}$ of $G$.
			
			We can thus take $\Gamma$-coinvariance to obtain a map
			$$
			X_*(T_g)_{\Gamma} \to \pi_1(G)_{\Gamma}
			$$
			and define
			$$H_g:=\ker(X_*(T_g)_{\Gamma} \to \pi_1(G)_{\Gamma}).$$

           With the choice of Borel containing $(T_g)_{\overline{F}}$, we have an isomorphism of abelian groups
           $$X_*((T_g)_{\operatorname{sc}}) \cong \Lambda_{\widehat{\Phi}},$$
           which is in general not $\Gamma$-equivariant.

			\begin{lemma}\label{Lemma:inflate}
				The image of the coroot lattice $\Lambda_{\widehat{\Phi}}$ under the composition
				$$\Lambda_{\widehat{\Phi}} \cong X_*((T_g)_{\operatorname{sc}}) \xhookrightarrow[]{}  X_*(T_g) \to X_*(T_g)_{\Gamma}$$
				lies in $H_g$. Thus we have an induced map
				$\Lambda_{\widehat{\Phi}} \to H_g.$
				In particular, any character $\chi: H_g \to \mathbb{C}^*$ can be viewed as a character of $\Lambda_{\widehat{\Phi}}$.
			\end{lemma}

			\begin{proof}
				The lemma follows from the following commutative diagram: 
				\begin{center}
					\begin{tikzcd}
						0 & {\Lambda_{\widehat{\Phi}}} & {X_*(T_g)} & {\pi_1(G)} & 0 \\
						&& {X_*(T_g)_{\Gamma}} & {\pi_1(G)_{\Gamma}}
						\arrow[from=1-1, to=1-2]
						\arrow[from=1-2, to=1-3]
						\arrow[from=1-3, to=1-4]
						\arrow[from=1-4, to=1-5]
						\arrow[from=1-3, to=2-3]
						\arrow[from=2-3, to=2-4]
						\arrow[from=1-4, to=2-4]
					\end{tikzcd}.
				\end{center}
			\end{proof}
			
			\begin{lemma}\label{Lemma:surj}
				The induced map $\Lambda_{\widehat{\Phi}} \to H_g$
				from Lemma \ref{Lemma:inflate} is surjective.
			\end{lemma}
			
			\begin{proof}
				The statement is set-theoretic, we can thus replace $\Lambda_{\widehat{\Phi}}$ by $X_*((T_g)_{\operatorname{sc}})$.
				
				Taking coinvariance is right exact, we thus have the following exact commutative diagram
				\[\begin{tikzcd}
					0 & {X_*((T_g)_{\operatorname{sc}})} & {X_*(T_g)} & {\pi_1(G)} & 0 \\
					& {X_*((T_g)_{\operatorname{sc}})_{\Gamma}} & {X_*(T_g)_{\Gamma}} & {\pi_1(G)_{\Gamma}} & 0
					\arrow[from=1-1, to=1-2]
					\arrow[from=1-2, to=1-3]
					\arrow[from=1-2, to=2-2]
					\arrow[from=1-3, to=1-4]
					\arrow[from=1-3, to=2-3]
					\arrow[from=1-4, to=1-5]
					\arrow[from=1-4, to=2-4]
					\arrow[from=2-2, to=2-3]
					\arrow[from=2-3, to=2-4]
					\arrow[from=2-4, to=2-5]
				\end{tikzcd}.\]
				This shows that the induced map $\Lambda_{\widehat{\Phi}} \cong X_*((T_g)_{\operatorname{sc}})\to H_g$ is surjective.
			\end{proof}
			
			\begin{lemma}\label{Lemma:fab}
				$H_g$ is a finite abelian group.
			\end{lemma}
			
			\begin{proof}
				Note in the proof of Lemma \ref{Lemma:surj} we have shown that the map $$X_*((T_g)_{\operatorname{sc}})_{\Gamma} \to H_g$$
				is surjective.  It thus suffices to show that $X_*((T_g)_{\operatorname{sc}})_{\Gamma}$ is finite. This follows from $T_g$ being elliptic.
			\end{proof}

			\subsection{Growth of $\chi(\character V_{\mu_m})$}

			Given a character $\chi: H_g \to \mathbb{C}^*$, by Lemma \ref{Lemma:inflate}, we can view it as a character $\chi: \Lambda_{\widehat{\Phi}} \to \mathbb{C}^*$ by inflation. Extending linearly, we can view it as a character
			$\chi: \mathbb{C}[\Lambda_{\widehat{\Phi}}] \to \mathbb{C},$
			where $\mathbb{C}[\Lambda_{\widehat{\Phi}}]$ is the group algebra of the abelian group $\Lambda_{\widehat{\Phi}}$ with $\mathbb{C}$-coefficients. In particular, it makes sense to talk about $\chi(\character(V_{\mu_m}))$ for $m \in \mathbb{Z}_{\geq 1}$.

			\begin{proposition} \label{Prop:dim}
				Let $k=
				|\widehat{\Phi}^+|$.
				\begin{enumerate}
					\item $\dim V_{\mu_m}$ is a polynomial in $m$ of degree $k$.
					\item Assume that $k \geq 1$. Then for any nontrivial character $\chi$ of $H_g$, $$\chi(\character(V_{\mu_m}))$$
					is bounded by a polynomial in $m$ of degree less or equal to $k-1$.
				\end{enumerate}
			\end{proposition}

			\begin{proof}
				Let $\chi$ be any character of $H_g$. By Lemma \ref{Lemma_char},

				\begin{align}    	&
					\chi(\character(V_{\mu_m}))  \\
					=\;& \prod_{{\alpha} \in \widehat{\Phi}^+}\chi(e^{2m{\alpha}}+e^{(2m-1){\alpha}}+...+e^{-2m{\alpha}})\\
					=\;&\label{Eq_prod} \prod_{{\alpha}\in \widehat{\Phi}^+}\chi({\alpha})^{-2m}(1+\chi({\alpha})+\chi({\alpha})^2+...+\chi({\alpha})^{4m}).
				\end{align}

				We focus on each term for $\alpha \in \widehat{\Phi}^+$. If $\chi({\alpha})=1$, then  
				$$\chi({\alpha})^{-2m}(1+\chi({\alpha})+\chi({\alpha})^2+...+\chi({\alpha})^{4m})=4m+1.$$
				If $\chi({\alpha}) \neq 1$, $\chi({\alpha})$ is some nontrivial root of unity since $H_g$ is finite abelian (Lemma \ref{Lemma:fab}). Therefore,
				$$\chi({\alpha})^{-2m}(1+\chi({\alpha})+\chi({\alpha})^2+...+\chi({\alpha})^{4m})$$
				is bounded as $m \to \infty$.
				
				We prove the first claim concerning $\dim V_{\mu_m}$. Let $\mathbbm{1}$ be the trivial character of $H_g$. By Lemma \ref{Lemma_char},
				$$\dim(V_{\mu_m})=\mathbbm{1}(\character(V_{\mu_m}))=\prod_{{\alpha} \in \widehat{\Phi}^+}(4m+1)=(4m+1)^k$$
				is indeed a polynomial in $m$ of degree $k$.
				
				Now we prove the second claim concerning $\chi(\character(V_{\mu_m}))$ for a nontrivial character $\chi$ of $H_g$. By Lemma \ref{Lemma:surj}, $\chi$ is non-trivial as a character of $\LPhi$. Since $\LPhi$ is spanned by $\widehat{\Phi}^+$ as an abelian group, there is some ${\beta} \in \widehat{\Phi}^+$ such that $\chi({\beta}) \neq 1$. Hence the corresponding (to ${\beta}$) term in Equation (\ref{Eq_prod}) is bounded. Each of the rest terms is either $(4m+1)$ or bounded.  Therefore, $\chi(\character(V_{\mu_m}))$ is bounded by a polynomial in $m$ of degree $\leq k-1$.
				
			\end{proof}

			\section{Proof of Theorem \ref{Thm_main_general}}
			
			\subsection{Hecke eigensheaves}
			
			Let $\varphi$ be an elliptic $L$-parameter, which defines a point in the stack of $L$-parameters. Let
			$$i_{\varphi}: [*/S_{\varphi}] \to [Z^1(W_F, \widehat{G})/\widehat{G}]$$
			be the inclusion. Recall that $\mathcal{O}(S_{\varphi})$ is the regular representation of the $\Ql$-algebraic group
			$S_{\varphi}$ (Subsection \ref{Subsection_regular_rep}). While $\OS$ is not a perfect complex on $[*/S_{\varphi}]$ in general, its isotypic summands $\OS_{\chi}$ are perfect complexes for any $\chi \in X^*(\ZG)$. The goal of this section is to show that for any sheaf $A \in \operatorname{D}(\Bun_G)^{\omega}$,
			$$i_{\varphi*}\mathcal{O}(S_{\varphi})*A:=\bigoplus_{\chi \in X^*(\ZG)}(i_{\varphi*}\OS_{\chi}*A)$$
			is a Hecke eigensheaf, i.e., for any finite-dimensional representation $V \in \Rep_{\widehat{G}}(\Ql)$,
			$$T_V(i_{\varphi*}\OS*A)=\left(i_{\varphi*}\OS*A\right)^{\oplus\dim V}.$$

			\begin{lemma}\label{Lemma:regular_rep}
				For any finite dimensional algebraic representation $W$ of $S_{\varphi}$, we have an isomorphism as $S_{\varphi}$-representations $$W \otimes \OS \cong \OS^{\oplus\dim W}.$$
			\end{lemma}
			
			\begin{proof}
				Let $*=\Spec \Ql$. We view algebraic representations of $S_\varphi$ as quasicoherent sheaves on the classifying stack $[*/S_{\varphi}]$. Let $\pr: * \to [*/S_{\varphi}]$ be the quotient map. Then the regular representation corresponds to the pushforward of the structure sheaf
				$$\OS=\pr_*\Ql.$$
				By projection formula,
				$$W \otimes \OS=W \otimes \pr_*\Ql \cong \pr_*(\pr^*W \otimes \Ql)=\pr_*(\Ql^{\oplus \dim W})=\OS^{\oplus\dim W}.$$
			\end{proof}

			\begin{proposition}\label{Prop:Hecke_eigen}
				For any $A \in \operatorname{D}(\Bun_G)^{\omega}$, $i_{\varphi*}\OS*A$ is a Hecke eigensheaf.
			\end{proposition}
			
			\begin{proof}
				Consider the maps 
				\begin{tikzcd}
					{[*/S_{\varphi}]} & {[Z^1(W_F, \widehat{G})/\widehat{G}]} & {[*/\widehat{G}]}
					\arrow["{i_{\varphi}}", from=1-1, to=1-2]
					\arrow["f", from=1-2, to=1-3]
				\end{tikzcd}, where $f$ is the natural projection from $[Z^1(W_F, \widehat{G})/\widehat{G}]$ to $[*/\widehat{G}]$. Note that the composition is the natural projection from $[*/S_{\varphi}]$ to $[*/\widehat{G}]$, so $i_{\varphi}^*f^*$ is the restriction from $\widehat{G}$ to $S_{\varphi}$.
				\begin{align*}    	
					& T_V(i_{\varphi*}\OS*A)  \\
					=\;&  f^*V * \left(i_{\varphi*}\OS*A\right) && \text{Compatibility of Hecke operators with spectral action}\\
					=\;& (f^*V \otimes i_{\varphi*}\OS)*A && \text{Spectral action is monoidal}&& 
					\\
					=\;& \left(i_{\varphi*}(i_{\varphi}^*f^*V \otimes \OS)\right)*A && \text{Projection formula}\\
					=\;& \left(i_{\varphi*}( \OS^{\oplus\dim V})\right)*A && \text{Lemma \ref{Lemma:regular_rep}}\\
					=\;& \left(i_{\varphi*}\OS*A\right)^{\oplus\dim V}. 
				\end{align*}
				
			\end{proof}

			\subsection{The Harish-Chandra  character $\Theta_{\pi_0}$ as a weighted sum}
			Recall that $\varphi$ is an elliptic $L$-parameter. For every $\pi \in \Irr_{\Ql} G(F)$ such that $\varphi_{\pi}^{\FS}=\varphi$, consider \begin{equation}\label{Notation_F0}
				\mathcal{F}_0:=(i_{\varphi})_*\mathcal{O}\left(S_{\varphi}/Z(\widehat{G})^{\Gamma}\right)*(i_1)_!\pi .
			\end{equation}
			Let $\pi_0:=i_1^*\mathcal{F}_0.$ Note that $\pi_0 \in \bigoplus_{b \in B(G)_{\basic}}\operatorname{D}(G_b(F), \Ql)$ is a direct sum of finitely many irreducible representations up to degree shifts (Proposition \ref{Proposition_elliptic}). In particular, the Harish-Chandra character $\Theta_{\pi_0}$ is well defined (Definition \ref{Def_HC}).
			For any $g \in G(F)_{\elliptic}$, the goal of this section is to express the Harish-Chandra character $\Theta_{\pi_0}(g)$ as a weighted sum over the stable conjugacy class of $g$ modulo $G(F)$-conjugacy. This is done by the following equation on Harish-Chandra characters
			\begin{equation}\label{Eq_Hecke_HC}
				\mathcal{T}_{V_{\mu_m}}\Theta_{\pi_0}(g)=\Theta_{i_1^*T_{V_{\mu_m}}(i_1)_*\pi_0}(g)=\dim(V_{\mu_m})\Theta_{\pi_0}(g),
			\end{equation}
			where $\mathcal{T}_{V_{\mu_m}}:=T_{1, \mu_m}^{G \to G}$ is the Hecke transfer map defined in \cite[Definition 3.2.7]{hansen2022kottwitz}. 
			
			First, let us prove the first equality of Equation (\ref{Eq_Hecke_HC}).
			
			\begin{lemma}\label{Lemma_first_eq}
				For any $g \in G(F)_{\elliptic}$,
				$$\mathcal{T}_{V_{\mu_m}}\Theta_{\pi_0}(g)=\Theta_{i_1^*T_{V_{\mu_m}}(i_1)_*\pi_0}(g).$$
			\end{lemma}
			
			\begin{proof}
				This follows from \cite[Theorem 6.4.5 and Theorem 6.5.2]{hansen2022kottwitz} and Lemma \ref{Lemma_selfdual}.
			\end{proof}

			Now let us turn to the second equality of (\ref{Eq_Hecke_HC}). Consider the sheaf 
			$$\mathcal{F}:=(i_{\varphi})_*\mathcal{O}(S_{\varphi})*(i_1)_!\pi.$$
			By Proposition \ref{Prop:Hecke_eigen}, $\mathcal{F}$ is a Hecke eigensheaf. 
			
			\begin{lemma}\label{Lemma_F_0_Hecke_eigen}
				For any $m \in \mathbb{Z}_{\geq 1}$, $\mathcal{F}_0$ is a Hecke eigensheaf \textbf{with respect to $T_{V_{\mu_m}}$}, i.e.
				$$T_{V_{\mu_m}}(\mathcal{F}_0)\cong \mathcal{F}_0^{\oplus \dim(V_{\mu_m})}.$$
			\end{lemma}
			
			\begin{proof}
				This follows from Proposition \ref{Prop:Hecke_eigen} and Lemma \ref{Lemma_pi_1} by noticing that the image of $\mu_m \in \Lambda_{\widehat{\Phi}}$ in $\pi_1(G)_{\Gamma}$ is trivial (Lemma \ref{Lemma:inflate}).			
			\end{proof}
			
			\begin{remark}
				As the Hecke operators $T_V$ will shift around the strata in general, $\mathcal{F}_0$ is not a Hecke eigensheaf for all $T_V, V \in \Rep_{\widehat{G}}(\Ql)$.
			\end{remark}

			\begin{lemma}
				$$\Theta_{i_1^*T_{V_{\mu_m}}(i_1)_*\pi_0}=\dim(V_{\mu_m})\Theta_{\pi_0}.$$
			\end{lemma}
			
			\begin{proof}
				This follows from Lemma \ref{Lemma_F_0_Hecke_eigen} and Lemma \ref{Lemma:*!}.
			\end{proof}
			
			Now we have finished the proof of Equation (\ref{Eq_Hecke_HC}).

			A corollary of Equation (\ref{Eq_Hecke_HC}) is an explicit formula for $\Theta_{\pi_0}$. To state this, let us recall some definitions from \cite{hansen2022kottwitz}. Let $g \in G(F)_{\elliptic}$ and $T_g=\Cent(g, G)$ with a choice of Borel over $\overline{F}$ containing $(T_g)_{\overline{F}}$. For $\lambda \in X_*(T_g)$ and $V \in \Rep_{\widehat{G}}(\Ql)$, let $\dim V[\lambda]$ be the weight multiplicity of $\lambda$ in $V$. Moreover, for $\lambda \in X_*(T_g)$, let $\overline{\lambda}$ be the image of $\lambda$ in $X_*(T_g)_{\Gamma} \cong B(T_g)$.
			
			\begin{definition}[{\cite[Definition 3.2.2]{hansen2022kottwitz}}]\label{Defition_inv}
				Let $b \in G(\breve{F})$ be a basic element. Let $g_1 \in G(F)_{\sr}$ and $g_b \in G_b(F)_{\sr}$ such that $g_1$ is conjugate to $g_b$ in $G(\overline{F})$. Let $T_1:=\Cent(g_1, G)$. Define $\inv[b](g_1, g_b)$ to be  the class of $y^{-1}by^{\sigma}$ in $B(T_1) \cong X_*(T_1)_{\Gamma}$, where $y \in G(\breve{F})$ satisfies $g_b=yg_1y^{-1}$. For $b=1$, we will denote 
				$\inv(g_1, g_b):=\inv[1](g_1, g_b).$
			\end{definition}

			We record some immediate observations which will be used later.
			
			\begin{lemma}\label{Lemma:ker}
				For $b=1$, $\inv(g_1, g_b) \in \ker(B(T_1) \to B(G)).$
			\end{lemma}
			
			\begin{proof}
				Note that $y^{-1}y^{\sigma}=y^{-1}.1.y^{\sigma}$ lies in the $\sigma$-conjugacy class of $1 \in G(\breve{F})$.
			\end{proof}
			
		\begin{lemma}\label{Lemma_inv_surj}
					For any $g \in G(F)_{\elliptic}$ and $T_g=\Cent(g, G)$. Denote the natural map $B(T_g) \to B(G)$ by $\alpha$. Then for any $b \in B(G)_{\basic}$,
					$$\inv[b](g, -): \{g' \in G_b(F) \;|\; g' \;\text{is stably conjugate to}\; g\} \to B(T_g)$$
					factors through $\alpha^{-1}(b) \subseteq B(T_g)$ and is surjective onto $\alpha^{-1}(b)$.
			\end{lemma}
			
			\begin{proof}
				First of all, note that by definition,
				$$\inv[b](g, g')=y^{-1}by^{\sigma}$$
				is $\sigma$-conjugate to $b$, hence $\alpha(\inv[b](g, g'))=b$, i.e., the map $\inv[b](g, -)$ factors through $\alpha^{-1}(b)$.
				
				It remains to show surjectivity. Indeed, let $[t] \in B(T_g)$ (choosing a lift $t \in T_g(\breve{F})$) such that $\alpha([t])=b$, i.e., there exists $h \in G(\breve{F})$ such that $hth^{-\sigma}=b$. We define
				$$g':=hgh^{-1}.$$
				Then $\inv[b](g, g')=h^{-1}bh^\sigma=t \in T_g(\breve{F})$. To show $g' \in G_b(F)$, we need to show that 
				$$g'^{\sigma}=b^{-1}g'b,$$
				which follows from $h^{-1}bh^{\sigma}=t \in T_g(\breve{F})=\Cent(g, G)(\breve{F})$. 
			\end{proof}
			
			\begin{corollary}\label{Corollary:explicit}
				For any $g \in G(F)_{\elliptic}$ and $T_g=\Cent(g, G)$ with a choice of Borel over $\overline{F}$ containing $(T_g)_{\overline{F}}$,
				\begin{equation}\label{Equation:sum_stable}
					\Theta_{\pi_0}(g)=\sum_{g' \in [[g]]}\frac{\sum_{\lambda \in X_*(T_g),\;\overline{\lambda}=\inv(g, g')}\dim V_{\mu_m}[\lambda]}{\dim V_{\mu_m}}\Theta_{\pi_0}(g'),
				\end{equation}
				where $[[g]]:=\{g' \in G(F)\;|\;g'\; \text{is conjugate to g in}\;  G(\overline{F})\}/G(F)-\text{conjugacy}$.
			\end{corollary}
			
			\begin{proof}
				This is clear from Equation (\ref{Eq_Hecke_HC}) by unraveling the definition of $\mathcal{T}_{V_{\mu_m}}=T_{1, \mu_m}^{G \to G}$ (\cite[Definition 3.2.7]{hansen2022kottwitz}). Note that the sign $(-1)^d$ in \cite[Definition 3.2.7]{hansen2022kottwitz} is always $1$ in our case because we take $\mu=\mu_m=4m\rho_G$.
			\end{proof}
			
			Since the right hand side of Equation (\ref{Equation:sum_stable}) is summing over $[[g]]$, the $G(\overline{F})$-conjugacy class modulo $G(F)$-conjugacy, we see that to show $\Theta_{\pi_0}$ is stable (i.e., invariant under $G(\overline{F})$-conjugation), it suffices to show that the coefficient in front of $\Theta_{\pi_0}(g')$ is essentially a constant independent of $g' \in [[g]]$. More precisely, in later sections, we will show that the limit
			$$\lim_{m \to \infty}\frac{\sum_{\lambda \in X_*(T_g),\;\overline{\lambda}=\inv(g, g')}\dim V_{\mu_m}[\lambda]}{\dim V_{\mu_m}}$$
			exists and is independent on $g' \in [[g]]$.

			\subsection{Equi-distribution of weight multiplicities}\label{Section_equi_dist}
			
			Recall that $g \in G(F)_{\elliptic}$ and $T_g=\Cent(g, G)$ with a choice of Borel over $\overline{F}$ containing $(T_g)_{\overline{F}}$. The goal of this section is to show that the limit
			$$\lim_{m \to \infty}\frac{\sum_{\lambda \in X_*(T_g),\;\overline{\lambda}=\inv(g, g')}\dim V_{\mu_m}[\lambda]}{\dim V_{\mu_m}}$$
			exists and is independent on $g' \in [[g]]$, hence complete the proof that $\Theta_{\pi_0}$ is stable.
			
			Recall $\pi_1(G):=X^*(\widehat{T})/\Lambda_{\widehat{\Phi}} \cong X_*(T_g)/\Lambda_{\widehat{\Phi}}$. By Lemma \ref{Lemma:ker}, $\inv(g, g') $ is mapped to $1$ under the composition 
			$$B(T_g) \cong X_*(T_g)_{\Gamma} \to \pi_1(G)_{\Gamma} \cong B(G)_{\basic}.$$
			In other words,  
			$$\inv(g, g')  \in H_g:=\ker(X_*(T_g)_{\Gamma} \to \pi_1(G)_{\Gamma}).$$ To obtain our goal, we show more generally that for any $h \in H_g$,
			$$\lim_{m \to \infty}\frac{\sum_{\lambda \in X_*(T_g),\;\overline{\lambda}=h}\dim V_{\mu_m}[\lambda]}{\dim V_{\mu_m}}$$
			exists and is independent on $h \in H_g$.
			
			For any $h \in H_g$ and $m \in \mathbb{Z}_{\geq 1}$, let
			$$S_{h, m}:=\frac{\sum_{\lambda \in X_*(T_g), \overline{\lambda}=h \in H_g} \dim (V_{\mu_m}[\lambda])}{\dim V_{\mu_m}}=\frac{\sum_{\overline{\lambda}=h} \dim (V_{\mu_m}[\lambda])}{\dim V_{\mu_m}}.$$
			
			\begin{theorem}\label{Thm:equi-dist}
				For any $h, h' \in H_g=\ker(X_*(T_g)_{\Gamma} \to \pi_1(G)_{\Gamma})$,
				$$\lim_{m \to \infty}(S_{h,m}-S_{h', m})=0.$$
			\end{theorem}
			
			\begin{proof}
				If $G$ has no roots, $G$ is a torus and $G=T_g$. Hence the map $X_*(T_g)_{\Gamma} \to \pi_1(G)_{\Gamma}$ is an isomorphism and $H_g=\{1\}$, and there is nothing to prove. So we may assume that $k=|\widehat{\Phi}| \geq 1$, i.e., the assumption of the second item of Proposition \ref{Prop:dim} is satisfied.
				
				We will show more generally that for any tuple $(c(h))_{h \in H_g} \in \mathbb{C}^{|H_g|}$ such that $\sum_{h \in H_g}c(h)=0$,
				$$\lim_{m \to \infty}\sum_{h \in H_g}c(h)S_{h, m}=0.$$
				
				For any nontrivial character $\chi$ of $H_g$, we can view it as a character of $\LPhi$ via the map $\LPhi \to H_g$ (Lemma \ref{Lemma:inflate}). Moreover, we extend linearly to view it as a map
				$$\chi: \mathbb{C}[\LPhi] \to \mathbb{C},$$
				where $\mathbb{C}[\LPhi]$ denotes the group algebra of the abelian group $\LPhi$. We shall denote the basis of the group algebra $\mathbb{C}[\LPhi]$ by $e^{\lambda}$ for $\lambda \in \LPhi$. 
				
				Note that $\character(V_{\mu_m}) \in \LPhi$. Applying $\chi$ to $\character(V_{\mu_m})$, we get
				\begin{align}\label{Eq:rearrange}
					\chi(\character(V_{\mu_m}))&\;=\chi(\sum_{h \in H_g}\sum_{\overline{\lambda}=h}\dim V_{\mu_m}e^{\lambda}) \\
					&\;=\sum_{h \in H_g}\sum_{\overline{\lambda}=h}\dim V_{\mu_m}\chi(\lambda)
					=\sum_{h \in H_g}\sum_{\overline{\lambda}=h}\dim V_{\mu_m}\chi(h).
				\end{align}
				
				So we have 
				$$\frac{\chi(\character(V_{\mu_m}))}{\dim V_{\mu_m}}=\sum_{h \in H_g}\frac{\sum_{\overline{\lambda}=h} \dim V_{\mu_m}[\lambda]}{\dim V_{\mu_m}}\chi(h)=\sum_{h \in H_g} \chi(h)S_{h, m}.$$
				
				By Proposition \ref{Prop:dim},
				$$\lim_{m \to \infty}\frac{\chi(\character(V_{\mu_m}))}{\dim V_{\mu_m}}=0,$$
				as inside the limit, the denominator is a polynomial in $m$ of degree $k$, and the numerator is bounded by a polynomial in $m$ of degree $\leq k-1$.
				
				Finally, notice that since $H_g$ is finite abelian (Lemma \ref{Lemma:fab}), the set of tuples
				$$\{(\chi(h))_{h \in H_g} \in \mathbb{C}^{|H_g|}\;|\;\chi\; \text{nontrivial character of}\;H_g\}$$
				spans the linear subspace
				$$\{(c(h))_{h \in H_g} \in \mathbb{C}^{|H_g|}\;|\;\sum_{h \in H_g}c(h)=0\}.$$ Therefore, for any tuple $(c(h))_{h \in H_g} \in \mathbb{C}^{|H_g|}$ such that $\sum_{h \in H_g}c(h)=0$,
				$$\lim_{m \to \infty}\sum_{h \in H_g}c(h)S_{h, m}=0,$$
				as desired.
			\end{proof}

			\begin{theorem}\label{Thm_indep}
				For any $h \in H_g$,
				The limit
				$$\lim_{m \to \infty}\frac{\sum_{\lambda \in X_*(T_g),\;\overline{\lambda}=h}\dim V_{\mu_m}[\lambda]}{\dim V_{\mu_m}}$$
				exists and is independent on $h \in H_g$.
			\end{theorem}
			
			\begin{proof}
				Note that for any $m \in \mathbb{Z}_{\geq 1}$,
				$$\sum_{h \in H_g}S_{h, m}=1.$$ Combining with Theorem \ref{Thm:equi-dist}, we see that for any $h \in H_g$,
				$$\lim_{m \to \infty}S_{h, m}=\frac{1}{|H_g|}$$ is independent of $h \in H_g$.
			\end{proof}
			
			Combing Corollary \ref{Corollary:explicit} with Theorem \ref{Thm_indep}, we get our main result, as follows. We note moreover that $\Theta_{\pi_{0}}$ is nonzero, see Lemma \ref{Lem_nonzero} below.
			
			\begin{theorem}\label{Thm_main_general}
				Let $G$ be a connected reductive group over $F$. Let $\varphi: W_F \to \widehat{G}(\Ql)$ be an elliptic $L$-parameter. For every $\pi \in \Irr_{\Ql}(G(F))$ such that $\varphi_{\pi}^{\FS}=\varphi$, let
				$$\mathcal{F}_0:=(i_{\varphi})_*\mathcal{O}\left(S_{\varphi}/Z(\widehat{G})^{\Gamma}\right)*(i_1)_!\pi, \qquad \pi_0:=i_1^*\mathcal{F}_0.$$
				Then the Harish-Chandra character $\Theta_{\pi_0}$ is stable as a non-zero function on $G(F)_{\elliptic}$, i.e., invariant under $G(\overline{F})$-conjugacy. 
				
			\end{theorem}
			
			\begin{remark}
				We emphasize that a prior, there might not exist $\pi \in \Irr_{\Ql}G(F)$ such that $\varphi_{\pi}^{\FS}=\varphi$. If that is the case, the theorem is empty. We expect that there always exists such $\pi$. However, the existence of such $\pi$ is not known in general.
			\end{remark}

			A prior, the Harish-Chandra character $\Theta_{\pi_0}$ might be $0$. We show that it is not the case in the following lemma.

			\begin{lemma}\label{Lem_nonzero}
				The Harish-Chandra charater $\Theta_{\pi_0}$ in Theorem \ref{Thm_main_general} is non-zero. Indeed, $\Theta_{\pi}$ occurs in $\Theta_{\pi_0}$ with a non-zero coefficient.
			\end{lemma}
			
			\begin{proof}
				Recall that $\pi_0=\OSZ*\pi$ is a finite direct sum of irreducible representations up to degree shifts. Note that $\pi=\operatorname{triv}\;*\;\pi$ occurs in $\pi_{0}$, where $\operatorname{triv}$ denotes the trivial representation of $S_{\varphi}$.
				Since the Harish-Chandra characters for different irreducible representations are linearly independent, the only possibility for $\Theta_{\pi_0}$ to be $0$ is that $\pi$ also occurs in some non-zero degree. 
				
				We shall prove this is not the case. By contradiction, $\pi[d_0]$ occurs in $\pi_0$ for some integer $d_0 \neq 0$. Then if we act by $\OSZ$ again we see that $\pi[nd_0]$ will appear in $\OSZ^{\otimes n} * \pi$ for any $n \in \mathbb{Z}_{\geq 1}$. However, by Lemma \ref{Lemma:regular_rep}, $$\OSZ^{\otimes n}=\OSZ^{\oplus e^{(n-1)}},$$ where $e=\dim \OSZ$. Thus $\OSZ^{\otimes n} * \pi$ has bounded degree uniform in $n$, which contradicts with $\pi[nd_0]$ appearing in $$\OSZ^{\otimes n} * \pi$$
				 for any $n \in \mathbb{Z}_{\geq 0}$.
			\end{proof}
			
			It was pointed out to us by David Hansen and later also by Wen-Wei Li that for elliptic virtual characters (including unitary discrete series virtual characters) in the sense of \cite{arthur1996local}, stability on $G(F)_{\elliptic}$ implies stability, as follows.
			
			\begin{theorem}\label{Thm_main_distribution}
				Under the same assumptions of Theorem \ref{Thm_main_general}, the Harish-Chandra character $\Theta_{\pi_0}$ is a nonzero stable distribution on $G(F)$.
			\end{theorem}
			
			\begin{proof}
			 By the proof of Proposition \ref{Proposition_elliptic}, the irreducible subquotients of $\pi_{0}$ are (up to degree shifts) supercuspidal representations with the same central character. Therefore,  $\Theta_{\pi_{0}}$ is a discrete series virtual character. Up to twist by a character of $G(F)$, $\Theta_{\pi_{0}}$ is a unitary discrete series virtual character, hence a elliptic virtual character in the sense of \cite{arthur1996local}. Now the theorem follows from \cite[Theorem 6.1]{arthur1996local}. See also \cite[Th. XI.3]{moeglin2016Stabilisation} and \cite[Proposition 3.2.6]{varmasome}.
			\end{proof}

			\subsection{Transfer between extended pure inner forms}

			Recall that we have proved in Theorem \ref{Thm_main_general} that the values of $\Theta_{\pi_0}$ on two stably conjugated elliptic elements of $G(F)$ are the same. In fact, we can even relate the character values of two stably conjugated elliptic elements of different extended pure inner forms.
			
			Recall that 
			$$\mathcal{F} :=\OS * \pi \cong \bigoplus_{\chi \in X^*(Z(\widehat{G})^{\Gamma})}(i_{\varphi})_*\mathcal{O}(S_{\varphi})_{\chi}*(i_1)_!\pi = \bigoplus_{\chi \in X^*(Z(\widehat{G})^{\Gamma})}\mathcal{F}_{\chi},$$
			where $\mathcal{F}_{\chi}:=(i_{\varphi})_*\mathcal{O}(S_{\varphi})_{\chi}*(i_1)_!\pi $
			is supported on $\Bun_G^{-\chi}$ by Lemma \ref{Lemma_pi_1}. 
			
			We shall first extend Lemma \ref{Lemma_first_eq} to arbitrary $\mu \in X_*$.	
			\begin{lemma}For  any $\mu \in X_*$,
				$$\mathcal{T}_{V_{\mu}}\Theta_{\mathcal{F}}=\Theta_{T_{V_{\mu}^{\vee}}\mathcal{F}}$$
				 as elements in $\prod_{b \in B(G)_{\basic}} C_c^{\infty}(G_b(F)_{\elliptic})$.
			\end{lemma}
			\begin{proof}
				It suffices to prove the similar equality for $\mathcal{F_{\chi}}$ instead of $\mathcal{F}$ for each $\chi \in X^*(\ZG)$. The equality holds in the component $C_c^{\infty}(G_{b-\chi}(F)_{\elliptic})$ for $b \in B(G, \mu)$ by \cite[Lemma 6.4.5 and Lemma 6.5.2]{hansen2022kottwitz}. For $b \notin B(G, \mu)$, the image of both sides in $C_c^{\infty}(G_{b-\chi}(F)_{\elliptic})$ is zero by  \cite[Lemma 3.2.8]{hansen2022kottwitz} and Lemma \ref{Lemma_pi_1}.
			\end{proof}
			
			$\mathcal{F}$ is a Hecke eigensheaf by Proposition \ref{Prop:Hecke_eigen}.
			Let $g \in G(F)_{\elliptic}$ and $T_g=\Cent(g, G)$ with a chosen Borel over $\overline{F}$ containing $(T_g)_{\overline{F}}$. For any $\chi \in X^*(\ZG)$, let $\mu \in X^*(\widehat{T}) \cong X_*(T_g)$ such that its image under
			$$X_*(T_g) \to X_*(T_g)_{\Gamma} \cong B(T_g) \to B(G)$$
			is $\chi$ (such $\mu$ exists by Theorem \ref{Thm_factor_basic}).

			We have
			\begin{equation}\label{Eq_inner_form}
				\mathcal{T}_{V_{\mu}}\Theta_{\mathcal{F}}=\Theta_{T_{V_{\mu}^{\vee}}\mathcal{F}}=\dim(V_{\mu}) \Theta_{\mathcal{F}}.
			\end{equation}
			Since for any $\chi' \in X^*(\ZG)$,
			$\mathcal{T}_{V_{\mu}}$
			transfers  functions on $G_{\chi'}(F)_{\elliptic}$ to  functions on $G_{\chi'-\chi}(F)_{\elliptic}$, if we look at the $b=-\chi$ part of Equation (\ref{Eq_inner_form}), we have
			\begin{equation}
				\mathcal{T}_{V_{\mu}}\Theta_{\mathcal{F}_0}=\dim(V_{\mu}) \Theta_{\mathcal{F}_{\chi}}.
			\end{equation}
			Now, for any $g_{-\chi} \in G_{-\chi}(F)_{\elliptic}$ and $T_{g_{-\chi}}:=\Cent(g_{-\chi}, G_{-\chi})$ with a chosen Borel over $\overline{F}$ containing $(T_{g_{-\chi}})_{\overline{F}}$, \begin{equation}\label{Equation_T_V_mu}
				(\mathcal{T}_{V_{\mu}}\Theta_{\mathcal{F}_0})(g_{-\chi}) = (-1)^d\sum_{g'}\sum_{\lambda}\dim V_{\mu}[\lambda]\Theta_{\mathcal{F}_0}(g')
			\end{equation} 
			where $d=\langle \mu, 2\rho_{G}\rangle$, $g'$ runs over elements in $G(F)$ (up to $G(F)$-conjugacy) that are stably conjugate to $g_{-\chi}$, and $\lambda$ runs over $X_*(T_{g_{-\chi}})$ such that $\overline{\lambda}=\inv[\chi](g_{-\chi}, g') \in X_*(T_{g_{-\chi}})_{\Gamma}$. Since we have shown that $\Theta_{\mathcal{F}_0}$ is stable, the terms
			$$\Theta_{\mathcal{F}_0}(g')$$
			in Equation (\ref{Equation_T_V_mu}) are independent on $g'$. Moreover, note that when $g' \in G(F)$ runs over elements that are stably conjugate to the given element $g_{-\chi} \in G_{-\chi}(F)_{\elliptic}$, $\lambda$ runs over elements in $X_*(T_{g_{-\chi}})$ such that $\dim(V_{\mu}[\lambda]) \neq 0$, as follows. 
			
			\begin{lemma}\label{Lemma_dim_neq_0}
				Let $\mu \in X^*(\widehat{T}) \cong X_*(T_g)$ such that its restriction to $Z(\widehat{G})^{\Gamma}$ is $\chi \in X^*(Z(\widehat{G})^{\Gamma}) \cong \pi_1(G)_{\Gamma}$. Let $\lambda \in X^*(\widehat{T}) \cong X_*(T_{g_{-\chi}})$ such that $\dim(V_{\mu}[\lambda]) \neq 0$, then there is some $g' \in G(F)$ stably conjugate to $g_{-\chi}$ such that 
				$$\inv[\chi](g_{-\chi}, g')=\overline{\lambda} \in X_*(T_{g_{-\chi}})_{\Gamma}.$$
			\end{lemma}
			
			\begin{proof}
				Recall that we have maps 
				$$X_*(T_{g_{-\chi}})_{\Gamma} \cong B(T_{g_{-\chi}}) \to B(G_{-\chi}) \cong B(G),$$
				where the image of the map in the middle is  $B(G)_{\basic}$ by Theorem \ref{Thm_factor_basic}. By Lemma \ref{Lemma_inv_surj}, it suffices to show that the image of $\lambda$ in $B(G)_{\basic}$ is $\chi \in X^*(Z(\widehat{G})^{\Gamma})$. Indeed, by highest weight theory, $\dim(V_{\mu}[\lambda]) \neq 0$ implies that the image of $\lambda$ and $\mu$ in $X^*(\widehat{T})/\Lambda_{\widehat{\Phi}} \cong \pi_1(G)$ are the same. Therefore, the image of $\lambda$ and $\mu$ in $\pi_1(G)_{\Gamma} \cong B(G)_{\basic} \cong X^*(Z(\widehat{G})^{\Gamma})$ are the same, which is $\chi$ by the assumption on $\mu$.
			\end{proof}

			So it follows from Equation (\ref{Equation_T_V_mu}) that for any $g' \in G(F)$ stably conjugate to $g_{-\chi} \in G_{-\chi}(F)_{\elliptic}$, we have
			\begin{equation}
				\Theta_{\mathcal{F}_0}(g')=(-1)^d\Theta_{\mathcal{F}_{\chi}}(g_{-\chi}),
			\end{equation}
			where $d=\langle \mu, 2\rho_{G}\rangle$. Note that $(-1)^d$ is independent of the choice of $\mu$ lifting $\chi$ because any two lifts differ by a $\mathbb{Z}$-linear combination of simple roots, and $\langle \alpha, 2\rho_{G}\rangle$ is even for any simple root $\alpha$ (\cite[13.1, Lemma A]{humphreys1972introduction}).
			In other words, we have shown the following corollary of Theorem \ref{Thm_main_general}.
			
			\begin{corollary}
				Let $G$ be a connected reductive group over $F$ and $\chi \in X^*(\ZG)$. Let $\varphi: W_F \to \widehat{G}(\Ql)$ be an elliptic $L$-parameter. For every $\pi \in \Irr_{\Ql}(G(F))$ such that $\varphi_\pi^{\FS}=\varphi$, let
				$$\mathcal{F}_{\chi}:=(i_{\varphi})_*\OS_{\chi}*(i_1)_!\pi, \qquad \pi_{\chi}:=i_{-\chi}^*\mathcal{F}_{\chi},$$
				where $i_{-\chi}$ denotes the open immersion $\Bun_G^{-\chi} \subseteq \Bun_G$. Given $g_{-\chi} \in G_{-\chi}(F)_{\elliptic}$ and $g' \in G(F)_{\elliptic}$ such that $g_{-\chi}$ and $g'$ are stably conjugate, we have
				$$\Theta_{\mathcal{F}_0}(g')=(-1)^d\Theta_{\mathcal{F}_{\chi}}(g_{-\chi}),$$
				where $d=\langle \mu, 2\rho_{G}\rangle$ for any $\mu \in X^*(\widehat{T})$ lifting $\chi$.
				
			\end{corollary}

			\bibliographystyle{alpha}
			\bibliography{reference}

		\end{document}